\documentclass[11p]{article}

\usepackage{times}
\usepackage{amsthm}
\usepackage{amsmath}
\usepackage{amssymb}
\usepackage{mathrsfs}
\usepackage{pb-diagram}
\usepackage{xcolor}
\usepackage{hyperref}

\def\o{\omega}
\def\O{\Omega}
\def\th{\theta}
\def\Th{\Theta}
\def\N{\mathbb{N}}

\def\Z{\mathbb{Z}}

\def\R{\mathbb{R}}
\def\Z{\mathbb Z}

\def\e{{\sf e}}

\def\r{{\rm r}}
\def\d{{\rm d}}
\def\bu{\bullet}

\def\({\left(}
\def\[{\left[}
\def\){\right)}
\def\]{\right]}

\def\si{\sigma}
\def\Si{\Sigma}
\def\G{\mathcal G}

\def\<{\langle}
\def\>{\rangle}

\usepackage{tikz-cd} 

 \newtheorem{thm}{Theorem}[section]
 \newtheorem{cor}[thm]{Corollary}
 \newtheorem{lem}[thm]{Lemma}
 \newtheorem{prop}[thm]{Proposition}
 \theoremstyle{definition}
 \newtheorem{defn}[thm]{Definition}
 \theoremstyle{remark}
 \newtheorem{rem}[thm]{Remark}
 \newtheorem{ex}[thm]{Example}
 \numberwithin{equation}{section}

\numberwithin{equation}{section}

\begin{document}


\title{Topological Dynamics of Groupoid Actions}


\author{F. Flores and M. M\u antoiu
\footnote{
\newline
{M. M. has been supported by the Fondecyt Project 1200884. 
}}
}



\maketitle


\begin{abstract}
Some basic notions and results in Topological Dynamics are extended to continuous groupoid actions in topological spaces. We focus mainly on recurrence properties. Besides results that are analogous to the classical case of group actions, but which have to be put in the right setting, there are also new phenomena. Mostly for groupoids whose source map is not open (and there are many), some properties which were equivalent for group actions become distinct in this general framework; we illustrate this with various counterexamples. 
\end{abstract}

\tableofcontents

\section{Introduction}\label{introduction}

Classically, topological dynamics is understood as the study of group and semigroup actions on topological spaces. It is an important chapter of modern mathematics originating from physics and the theory of differential equations, and its theoretical and practical outreach need not be outlined here. The point of view we adopt is that of the abstract theory, as exposed in references such as \cite{Au,dV,EE,Gl}.

\smallskip
Basic to topological dynamics in the classical sense is the idea of global symmetry. However, many interesting systems only present {\it local (or partial) forms of symmetry}. Such systems have acquired an increasing role in modern mathematics, and their impact in applications is already largely acknowledged. 

\smallskip
Partial symmetry is treated using concepts as {\it groupoids}, {\it partial group actions} or {\it inverse semigroup actions}. Several aspects are already very well developed. In fact our interest was raised by the applications to $C^*$-algebras, as in \cite{Ex,Ex1,Pa,Re,Wi}. However, up to our knowledge, the basic theory in the spirit of classical topological dynamics has not yet been developed systematically. In the literature, when notions and results are needed, they are briefly introduced and used in an ad hoc way. In addition, certain basic concepts barely appear in the partial symmetry setting.

\smallskip
The present article is dedicated to a systematic study of the most elementary dynamical notions in the framework of {\it continuous groupoid actions on topological spaces}. We emphasize {\it recurrence phenomena}. Some connections with partial group actions are also made, but the interaction with inverse semigroup dynamical systems, as well as a deeper study of more advanced dynamical notions, are deferred to a subsequent article.

\smallskip
As a general convention, all the topological spaces (including the groupoids) are Hausdorff. Local compactness is required only when needed. In certain cases we ask all the fibers of the groupoid to be non-compact, in order to avoid triviality. 

\smallskip
In most cases (but not always), guessing the analog of the standard notions in the groupoid case is quite straightforward. Both in the statements and the proofs, one has to take into account the fibered nature of a groupoid action. When a groupoid $\Xi$ with unit space $X$ acts on a topological space $\Si$\,, the action $\th$ is only defined on a subset of the product $\Xi\!\times\!\Si$\,. The element $\th_\xi(\si)$ is defined only if the domain $\d(\xi)$ of $\xi\in\Xi$ coincides with the image $\rho(\si)$ of $\si\in\Si$ by the {\it anchor map} $\rho:\Si\to X$, which is part of the definition of the action. There are also some technical difficulties that we discuss briefly: 

\smallskip
(a) When dealing with $\th_\xi(\si)$\,, sometimes one needs to approximate $\xi$ by a net $\{\xi_i\!\mid\!i\in I\}$\,. Usually $\th_{\xi_i}\!(\si)$ makes no sense, and this requires more involved arguments.

\smallskip
(b) When a group $\G$ acts over a topological space $\Si$\,, if ${\sf A}\subset\G$ and $U\subset\Si$ are open, ${\sf A}\,U\subset\Si$ is also open. Rather often, this plays an important role in the proofs. For groupoid actions, this is true only in the special case when the domain (source) map $\d:\Xi\to\Xi$ is open; we say that the groupoid is {\it open} if this holds. This may fail even in simple examples. In such cases some expected properties do not hold, as counterexamples will show. 

\smallskip
(c) Even if the map $\d$ is open, in most cases the translation $\th_\xi(U)$ of an open set is no longer open. For group actions this phenomenon is absent, since the action is composed of (global) homeomorphisms. Therefore, extra care will be needed in some proofs.

\smallskip
Let us describe the content of the article.

\smallskip
{\bf Section \ref{goam}} introduces the basic definitions. To read the article, the reader only needs to know what a (topological) groupoid is, and to be familiar with the elementary properties a classical dynamical system might have. The notations involving groupoids and continuous groupoid actions are introduced in {\bf \ref{grefier}}. The canonical action on the unit space will be very important; it is a terminal action, by Lemma \ref{ajuta}. Notions as orbit, orbit closure, invariant set and saturation are obvious extensions of the standard ones. However, {\it if the groupoid $\Xi$ is not open, the interior, the closure or the boundary of an invariant subset of $\Si$ might not be invariant}. This will have far-reaching negative consequences. The central notion of {\it recurrence set} $\widetilde\Xi^N_M$ appears in {\bf \ref{calau}}. For $M,N\subset\Si$\,, this consists in the elements $\xi\in\Xi$ under which some point of $M$ is sent inside $N$ (taking into account domain issues).

\smallskip
{\bf Section \ref{cavaler}.} In certain situations, a groupoid (or a groupoid action) is canonically constructed from other related mathematical objects. We present here a couple of such situations, stressing the way recurrence sets are determined by these objects. Our selection is probably inspired by our interest in groupoid $C^*$-algebras; geometrically oriented readers could prefer others. {\it Equivalence relations} are important in groupoid theory; for us, they will often serve as counterexamples. In fact, the most general of our counterexamples are subgroupoids of direct products between equivalence relations and groups. In {\bf \ref{castyanu}}, the usual {\it group actions} are encoded in two different ways to fit our setting. It is shown that the second way yields a better grasp at the level of recurrence. {\it The Deaconu-Renault groupoid} of a local homeomorphism appears in {\bf \ref{caterinka}}. In {\bf \ref{creasta}} and {\bf \ref{castanika}} we arrive at groupoid actions making {\it non-invariant restrictions} in larger, maybe global, dynamical systems; this includes partial group actions. In {\bf \ref{castanier}}, to a groupoid action $(\Xi,\th,\Si)$ we associate a larger {\it transformation (crossed product) groupoid} $\Xi\!\ltimes_\th\!\Si$\,. It only has a partial relevance for recurrence issues. But it will also be used later, in connection with factors and extensions. A {\it pull-back construction} appears in {\bf \ref{cabalier}}.

\smallskip
{\bf Section \ref{calorifiert}} mainly treats various types of {\it topological transitivity}. Consider the following properties of a continuous action $\th$ of the groupoid $\Xi$ in the topological space $\Si$\,:
$(i)$ $\Si$ is not the union of two proper invariant closed subsets, $(ii)$ non-empty open invariant subsets of $\,\Si$ are dense,
$(iii)$ for every $\,U,V\subset\Si$ open and non-void,  $\th_\xi(U)\cap V\ne\emptyset$ for some $\xi\in\Xi$\,,
$(iv)$ invariant subsets of $\,\Si$ are ether dense, or nowhere dense. For group actions, these properties (suitably reformulated) are equivalent and go under the common name of "topological transitivity". This also happens for groupoid actions, if $\Xi$ is open (i.\,e.\! the domain map is open). But if it is not, then one only has $(iv)\Rightarrow (iii)\Rightarrow (ii)\Rightarrow (i)$\,. After proving this and some connected results in Theorem \ref{pricinoasa}, we indicate a couple of counterexamples showing that in general none of the implications can be reverted. It is also shown that pointwise transitivity (existence of a dense orbit) implies $(iii)$ but not $(iv)$. The failure of some implications can be tracked back to the fact that the closure of an invariant set could not be invariant if the groupoid is not open. We decided to call $(iii)$ {\it recurrent transitivity} and $(iv)$ (the strongest notion) {\it topological transitivity}; it is debatable whether the terminology is the best one. Properties $(i)$ and $(ii)$ did not receive a name. Proposition \ref{uncorolar} makes the necessary specifications for the case of a Baire second-countable space $\Si$\,. Then we particularize to some of the examples from Section \ref{cavaler}. In {\bf \ref{pisica}} we introduce and study a final transitivity notion, {\it  the weak pointwise transitivity}. It is similar to pointwise transitivity, but the orbit closure of a point (that might not be invariant if $\Xi$ is not open) is replaced by the smallest closed invariant set containing the point.

\smallskip
{\bf Section \ref{cracifloc}} is concerned with limit points, recurrent points and wandering and non-wandering properties. 
{\it Limit sets} are mostly studied for actions of one of the groups $\Z$ or $\R$ on topological spaces, where one distinguishes between positive and negative limit points. A good reference is \cite[Chapter II.2]{dV}. We are going to adapt the basic part of the theory to groupoids. But besides being groups, $\Z$ and $\R$ have two extra features: (1) they have 'two ends' and (2) they are not compact. The first feature leads to various ramifications, as distinguishing between positive and negative limit sets. Trying to imitate this for groupoid actions is possible, but would require a specialized and rather intricate setting. On the other hand, too much compactness in a groupoid would make some constructions and results trivial or non-interesting. So in {\bf \ref{craciflor}}, we will restrict our attention to \textit{strongly non-compact groupoids}\,, those for which all the fibers $\Xi_x$ are non-compact. The limit points of a given point $\si$ are defined by their asymptotic behavior under the action, where "asymptotic" is encoded, equivalently, by complements of compact sets, diverging nets (both in $\Xi$) or suitable recurrence sets. They are contained in the orbit closure of $\si$ and form a closed subset having an attractor behavior. This subset is invariant if the groupoid is open, but a counterexample shows that in general this can fail. In {\bf \ref{craciclor}} a point is called {\it recurrent} if it is a limit point of itself. In Proposition \ref{recur} this is related to other types of descriptions, as in classical Dynamical Systems, but once again full equivalences hold under the openness assumption (a counterexample is presented). This assumption is also needed to insure invariance of the family of all recurrent points. {\it Wandering} and {\it non-wandering} are similar to the group case notions, but defined in terms of groupoid recurrence sets $\widetilde\Xi^W_W$ and compactness, where $W$ is a neighborhood of the point we study. The family of non-wandering points is closed; to be invariant one also requires $\Xi$ to be open and locally compact. If $\Si$ is compact, it is non-void and attracts the points of $\Si$ in a suitable sense. We indicate examples involving pull-backs and action groupoids.

\smallskip
{\bf Section \ref{calorifier}.} We start in {\bf \ref{golorofer}} with the set $\Si_{\rm fix}\subset\Si$ of {\it fixed points} of a groupoid action. It is invariant; it is also closed if $\Xi$ happens to be open, but not in general, as counterexamples show. We inspect the origin of the fixed points for global and partial group actions, for non-invariant restrictions, for the Deaconu-Renault groupoid and for pull-backs. In {\bf \ref{kalorifier}} we show the inclusion
\begin{equation}\label{pestich}
\Si_{\rm fix}\subset\Si_{\rm per}\subset\Si_{\rm wper}\cap\Si_{\rm alper}\subset\Si_{\rm wper}\cup\Si_{\rm alper}\subset\Si_{\rm rec}\subset\Si_{\rm nw}\,.
\end{equation} 
Besides fixed, recurrent and non-wandering points, already defined, we introduce three other types of points: {\it periodic}, {\it weakly periodic} and (most important) {\it almost periodic}. For the first two, terminology could be debated even in the group case. The precise meaning can be found in Definition \ref{indrasniesc}, where an adaptation to groupoids of the standard notion of {\it syndeticity} is also included. We discuss issues such as invariance and closure for the new sets.  In \cite{BBdN} the authors work with locally compact, second countable, open groupoids with compact unit space. A unit $x$ is called {\it a periodicoid} if its orbit is closed (i.e.\;compact). Among others they show that if, in addition, the groupoid is \'etale, the orbit of a periodicoid point is actually finite. Obviously, the notion makes sense and is relevant also for general groupoid actions. In {\bf \ref{coproduct}} we connect it with our periodicity. Proposition \ref{lacidraci} proves that a periodic point $\si$ has a compact orbit and in Proposition \ref{dracilaci} we prove the converse implication, adding some extra conditions. {\it Minimality} is explored in {\bf \ref{cralorifier}}. After stating the definition and providing the most elementary properties, we exhibit in Theorem \ref{sacacarez} the connection between minimality and almost periodicity. For group actions, this is the standard result that can be found in every textbook \cite{Au,dV,EE,Gl}. The proof is similar, but slightly more involved, because of the groupoid setting. Corollary \ref{oblu} deals with semisimplicity and point almost periodicity. In Proposition \ref{zdrikat} we provide conditions under which minimality implies the non-wandering property.

\smallskip
{\bf Section \ref{kalorfier}.} {\it Homomorphisms} (equivariant maps) are an important topic in Topological Dynamics. In {\bf \ref{ficatorel}} we study surjective homomorphisms (epimorphisms) in the framework of groupoid actions (equivariance now also requires a compatibility of the two anchor maps). They lead to the usual concepts of {\it extension} and {\it factor}. The canonical action of a groupoid $\Xi$ on the unit space is terminal, being a factor of any other $\Xi$-action. We study the fate under epimorphisms of most of the dynamical properties already introduced. Two main results are Theorem \ref{chiroare} (referring to the sets \eqref{pestich}) and Proposition \ref{morptranz} (treating various types of transitivity, including recurrent transitivity). Note that, by Example \ref{valioso4}, topological transitivity do not always transfer to factors when the groupoid is not open. Proposition \ref{garbanzos} refers to the behavior of minimality under epimorphisms. It also contains conditions under which a minimal subsystem of the factor admits a minimal pre-image and an almost periodic point of the factor is reached from an almost periodic element of the extension. Subsection {\bf \ref{fregadorel}} connects extensions of groupoid actions as in {\bf \ref{ficatorel}} with action (crossed product) groupoids as in {\bf \ref{castanier}}. In the recent preprint \cite{EK} the authors use transformation groupoids to study group action extensions. They develop a very interesting theory, that however has little in common with the content of the present article. By a straightforward generalization of a construction from \cite{EK}, given a {\it groupoid} action $\Th'\!=\big(\Xi,\th'\!,\Si'\big)$\,, we show that {\it there is a one-to-one correspondence between extensions of $\,\Th'$ and actions of the crossed product groupoid $\,\Xi(\Th')\!:=\Xi\,\ltimes_{\th'}\!\Si'$}. We also exhibit an isomorphism between two different crossed products. In Propositions \ref{vindeo} and \ref{deruta} we prove that the dynamical properties are the same under the mentioned bijective correspondence.

\smallskip
{\bf Section \ref{frigider}.}
We dedicate this short final section to convince the reader that mixing, as presented in Definition \ref{ixing}, is not an interesting concept outside the classical group case, if the anchor map $\rho:\Si\to X$ is surjective and the unit space $X$ is Hausdorff (standing assumptions in the present paper).

\section{Groupoids and groupoid actions}\label{goam}

\subsection{The framework}\label{grefier}

We deal with groupoids $\Xi$ over a unit space $\Xi^{(0)}\!\equiv X$, seen as small categories in which all the morphisms (arrows) are inverible. 
The source and range maps are denoted by ${\rm d,r}:\Xi\to \Xi^{(0)}$ and the family of composable pairs by $\,\Xi^{(2)}\!\subset\Xi\times\Xi$\,. For $M,N\subset X$ one uses the standard notations 
\begin{equation*}\label{faneaka}
\Xi_M\!:={\rm d}^{-1}(M)\,,\quad\Xi^N\!:={\rm r}^{-1}(M)\,,\quad\Xi_M^N:=\Xi_M\cap\Xi^N.
\end{equation*}
Particular cases are the $\d$-fiber $\Xi_x\equiv\Xi_{\{x\}}$\,, the  ${\rm r}$-fiber $\Xi^x\equiv\Xi^{\{x\}}$ and the isotropy group  $\Xi_x^x\equiv\Xi_{\{x\}}^{\{x\}}$ of a unit $x\in X$. Clearly $\Xi_x^x\,\Xi^x\!\subset\Xi^x$ and $\Xi_x\,\Xi_x^x\!\subset\Xi_x$\,. The disjoint union
\begin{equation*}\label{gigolo}
{\rm Iso}(\Xi):=\bigsqcup_{x\in X}\Xi_{x}^{x}=\{\xi\in\Xi\!\mid\! \d(\xi)=\r(\xi)\}
\end{equation*}
is called \textit{the isotropy bundle of the groupoid}.

\smallskip
The subset $\Delta$ of the topological groupoid $\Xi$ is called a {\it subgroupoid} if for every $(\xi,\eta)\in(\Delta\!\times\!\Delta)\cap\Xi^{(2)}$ one has $\xi\eta\in\Delta$\, and $\xi^{-1}\in\Delta$. This subgroupoid is {\it wide} if $\Delta^{(0)}=\Xi^{(0)}$. 

\smallskip
A topological groupoid is a groupoid $\Xi$ with a topology such that the inversion $\xi\mapsto\xi^{-1}$ and multiplication $(\xi,\eta)\mapsto \xi\eta$ are continuous. The topology in $\Xi^{(2)}$ is the topology induced by the product topology. Whenever the map $\d:\Xi\to X$ is open, we say that {\it the groupoid $\Xi$ is open}. Equivalent conditions are: (i) $\r:\Xi\to X$ is open and (ii) the multiplication is an open map. It is known that a locally compact groupoid possessing a Haar system is open. In particular, \'etale groupoids and Lie groupoids are open. For this reason, many texts in groupoid theory only deal with open groupoids.

\smallskip
An equivalence relation on $X$ is defined by $x\approx y$ if $x=\r(\xi)$ and $y={\rm d}(\xi)$ for some $\xi\in\Xi$\,. This leads to the usual notions of {\it orbit, invariant (saturated) set, saturation, transitivity}, etc. {\it The orbit of a point} $x$ will be denoted by $\mathcal O_x={\rm r}(\Xi_x)$\,. Its closure $\overline{\mathcal O}_x$ is called {\it an orbit closure}.

\begin{defn}\label{grupact}
{\it A groupoid action} is a 4-tuple $(\Xi,\rho,\th,\Si)$ consisting of a groupoid $\Xi$\,, a set $\Si$\,, a surjective map $\rho:\Si\rightarrow X$ ({\it the anchor}) and the action map
\begin{equation}\label{ganchor}
\th:\Xi\!\Join\!\Si:=\{(\xi,\si)\!\mid\!\d(\xi)=\rho(\si)\}\ni(\xi,\si)\mapsto{\th_\xi(\si)\equiv\xi\!\bu_\th\!\si}\in\Si
\end{equation} 
satisfying the axioms: 
\begin{enumerate}
\item $\rho(\si)\!\bu_\th\!\si=\si,\, \forall \,\si\in \Si$\,,
\item if $(\xi,\eta)\in \Xi^{(2)}$ and $(\eta,\si)\in \Xi\!\Join\!\Si$, then $(\xi,\eta\!\bu_\th\!\si)\in \Xi\!\Join\!\Si$ and $(\xi\eta)\!\bu_\th\!\si=\xi\!\bu_\th\!(\eta\!\bu_\th\!\si)$\,.
\end{enumerate} 
{\it An action of a topological groupoid} in a topological space is just an action $(\Xi,\rho,\th,\Si)$ where $\Xi$ is a topological groupoid, $\Si$ is a Hausdorff topological space and the maps $\rho,\th$ are continuous.
\end{defn}

If the action $\th$ is understood, we will write $\xi\bu\si$ instead of $\xi\bu_\th\si$. If $\rho$ is not supposed surjective, then $\rho(\Si)$ is an invariant subset of the unit space and only the reduction $\Xi_{\rho(\Si)}^{\rho(\Si)}$ really acts on $\Si$\,, so asking $\rho$ to be onto seems convenient.

\begin{ex}\label{startlet}
Each topological groupoid acts continuously in a canonical way on its unit space. In this case, we have $\Si=X$ and $\rho={\rm id}_X$, and then (note the special notation) $\xi\!\circ\!x:=\xi x\xi^{-1}$ as soon as $\d(\xi)=x$. Putting this differently, $\xi$ sends $\d(\xi)$ into $\r(\xi)$\,. One could also name this {\it the terminal action}; see Lemma \ref{ajuta}.
\end{ex}

\begin{ex}\label{novotel}
Let $(\Xi,\rho,\th,\Si)$ be a continuous groupoid action and $\Delta$ a wide subgroupoid of $\Xi$\,. The restricted action is defined by keeping the same anchor $\rho$ and just restricting the map $\th$ to
\begin{equation*}\label{prinparti}
\Delta\!\Join\!\Si=\{(\xi,\si)\in\Delta\!\times\!\Si\!\mid\!\d(\xi)=\rho(\si)\}=\{(\xi,\si)\in\Xi\!\Join\!\Si\!\mid\!\xi\in\Delta\}\,.
\end{equation*}
\end{ex}

\begin{ex}\label{valtoare}
The topological groupoid $\Xi$ also acts on itself, with $\Si:=\Xi$\,, $\rho:=\r$ and $\xi\bu\eta:=\xi\eta$\,.
\end{ex}

For $\xi\in\Xi\,,\,{\sf A},{\sf B}\subset\Xi\,,\,M\subset\Si$ we use the notations
\begin{equation}\label{botations}
{\sf A}{\sf B}:=\big\{\xi\eta\,\big\vert\,\xi\in{\sf A}\,,\,\eta\in{\sf B}\,,\,\d(\xi)=\r(\eta)\big\}\,,
\end{equation}
\begin{equation*}\label{ottations}
\xi\bu M:=\big\{\xi\bu\si\,\big\vert\,\si\in M\cap\rho^{-1}[\d(\xi)]\big\}\,,
\end{equation*}
\begin{equation*}\label{otattions}
{\sf A}\bu M:=\big\{\xi\bu\si\,\big\vert\,\xi\in{\sf A}\,,\si\in M\,, \d(\xi)=\rho(\si)\big\}=\bigcup_{\xi\in{\sf A}}\xi\bu M\,.
\end{equation*}
These sets could be void in non-trivial situations. 

\begin{rem}\label{caofi}
In \cite{Wi} it is shown that if the groupoid $\Xi$ is open then ${\sf A}\bu M\subset \Si$ is open whenever the sets ${\sf A}\subset\Xi$ and $M\subset \Si$ are open. For topological group actions  the product ${\sf A}M$ is open provided that only the subset $M$ is open. In addition, if ${\sf A},{\sf B}$ are subsets of the group, ${\sf A}{\sf B}$ is open whenever at least one of the subsets is. Examples below will show that for groupoids this is not true, and this will require some special care in some of our proofs. Note that, even for open groupoids, the translation $\xi\bu M$ of an open subset $M$ of $\Si$ could not be open.
\end{rem}

\begin{defn}\label{pinera}
We are going to use {\it orbits} $\mathfrak O_\si\!:=\Xi_{\rho(\si)}\!\bu\si$ and {\it orbit closures} $\overline{\mathfrak O}_\si$\,.  {\it The orbit equivalence relation} will be denoted by $\sim$\,. A subset $M\subset \Si$ is called {\it invariant} if $\xi \bu M\subset M$, for every $\xi\in\Xi$\,. If $N\subset\Si$\,, its {\it saturation} 
\begin{equation*}\label{satur}
{\sf Sat}(N)=\Xi\bu N=\!\!\bigcap\limits_{\substack{N\subset M \\ M\textup{ invariant}}}\!\!M
\end{equation*}
is the smallest invariant subset of $\Si$ containing $N$.
\end{defn}

\begin{prop}\label{intclo}
Assume that $\Xi$ is an open groupoid. The saturation of an open set is also open. The interior $M^\circ$, the closure $\overline M$ and the boundary $\partial M$ of an invariant subset $M$ of $\,\Si$ are also invariant. 
\end{prop}

\begin{proof}
If $N$ is an open set, ${\sf Sat}(N)=\Xi\bu N$ is open, as $\d$ is an open map; see Remark \ref{caofi}. 

\smallskip
One has $(M^\circ)^c=\overline{M^c}$ and $\partial M=\overline M\setminus M^\circ$. Since the difference of two invariant sets is clearly invariant, it is enough to show that $M^\circ$ is invariant. If $\sigma\in M$ is an interior point, there exists some open set $U\subset M$ containing $\si$, so we have 
$$
\xi\bu\si\in\Xi\bu U\subset\Xi\bu M=M,
$$ 
implying that $\xi\bu\si$ is also an interior point of $M$, since $\Xi\bu U$ is open. 
\end{proof}

Remark \ref{tutosh} and Example \ref{valioso} will show that the openness assumption cannot be removed. 

\begin{lem}\label{lavila}
For every $\si\in\Si$ one has $\rho\big(\mathfrak O_\si\big)=\mathcal O_{\rho(\si)}$\,. The map $\rho$ sends $\bu$-invariant subsets of $\,\Si$ into $\circ$-invariant subsets of  the unit space $X$ (see Example \ref{startlet}).
\end{lem}

\begin{proof}
If $y\in\rho\big(\mathfrak O_\si\big)$\,, for some $\xi\in\Xi_{\rho(\si)}$ one has
$$
y=\rho(\xi\bu\si)=\r(\xi)\in\mathcal O_{\rho(\si)}\,.
$$ 
On the other hand, if $y\in\mathcal O_{\rho(\si)}$\,, there exists $\xi\in\Xi_{\rho(\si)}$ with $y=\r(\xi)=\rho(\xi\bu\si)\in\rho\big(\mathfrak O_\si\big)$ and the equality is proven. From this, the last part is trivial. 
\end{proof}

\subsection{Recurrence sets}\label{calau}

\begin{defn}\label{giudat}
For every $M,N\subset\Si$ one defines {\it the recurrence set} as 
\begin{equation*}\label{rrecur}
\widetilde\Xi_M^N=\{\xi\in\Xi\!\mid\!(\xi\bu M)\cap N\ne\emptyset\}\,.
\end{equation*} The set $\widetilde\Xi_M^N$ is increasing in $M$ and $N$. In the group case, one also uses the term "dwelling set". 
\end{defn}  

\smallskip
The recurrence set can be also described in terms of the function
\begin{equation*}\label{badguy}
\Xi\!\Join\!\Si\ni(\xi,\si)\overset{\vartheta}{\to}(\xi\bu\si,\si)\in\Si\!\times\!\Si.
\end{equation*}
If we denote by $q$ the projection on the first variable $\Xi\!\times\!\Si\to\Xi$ and by $\mathfrak q$ its restriction to $\Xi\!\Join\!\Si$\,, then
\begin{equation*}\label{recurr}
\widetilde\Xi_M^N\!:=\mathfrak q\big[\vartheta^{-1}(N\!\times\!M)\big]\,.
\end{equation*}

\begin{rem}\label{iverse}
Note that $\,\widetilde\Xi_M^N=\bigcup_{\si\in M}\widetilde\Xi_\si^N$, where 
$$
\widetilde\Xi_\si^N\equiv\widetilde\Xi_{\{\si\}}^N\!=\{\xi\in\Xi_{\rho(\si)}\!\mid\!\xi\bu\si\in N\}\subset\Xi_{\rho(\si)}^{\rho(N)}\subset\Xi_{\rho(\si)}\,.
$$
It follows immediately that $\widetilde\Xi_M^N\subset\Xi_{\rho(M)}^{\rho(N)}$\,. Actually, when $\rho$ is also injective, one has $\,\widetilde\Xi_M^N=\Xi_{\rho(M)}^{\rho(N)}$\,. This applies, in particular, to Example \ref{startlet}. {\it The stabilizer} $\widetilde\Xi_\si^\si=\{\xi\in\Xi\!\mid\!\xi\bu\si=\si\}$ is a closed subgroup of the isotropy group $\Xi_{\rho(\si)}^{\rho(\si)}$\,. They coincide whenever $\rho$ is injective.
\end{rem}

\begin{ex}\label{labar}
In the setting of Example \ref{valtoare} one has
$$
\widetilde\Xi_M^N=\big\{\xi\in\Xi\,\big\vert\, \Xi^{\d(\xi)}\!\cap M\cap\xi^{-1}N\ne\emptyset\big\}\,.
$$
\end{ex}

\begin{ex}\label{lebar}
Given a topological space $\Si$\,, {\it the fundamental groupoid $\Xi$}\,, typically denoted by $\Pi_1(\Si)$\,, is just the set of homotopy classes of paths between pairs of points. The space of all paths is given the compact-open topology and this induces in $\Xi$ the quotient topology. We observe that $X$ is basically $\Si$ and the action $\xi\circ\si$ moves the starting point $\si$ through the path $\xi$\,. The recurrence sets are expressed in terms of the path-connectedness of $\Si$\,: 
$$
\widetilde\Xi_M^N=\Xi_M^N=\{\gamma:[0,1]\rightarrow\Si \mid \gamma \textup{ continuous },\gamma(0)\in M\textup{ and }\gamma(1)\in N\}\textup{ modulo homotopy.}
$$ 
 The stabilizer of $\si\in X$ is the fundamental group $\Xi_{\si}^{\si}\!=\pi(\Si,\si)$ of $\Si$ rooted at $\si$.
\end{ex}

The next straightforward results will be useful in the next sections.

\begin{lem}\label{ushor}
If $M,N\subset\Si$ and $\eta_1,\eta_2\in\Xi$\,, then $\widetilde\Xi_{\eta_1\bu M}^{\eta_2\bu N}=\eta_2\,\widetilde\Xi_M^N\eta_1^{-1}$ and $\big(\widetilde\Xi_M^N\big)^{-1}\!=\widetilde\Xi_N^M$\,. 
\end{lem}

\begin{proof}
We only show the first equality:
$$\begin{aligned}
\xi\in\widetilde\Xi_{\eta_1\bu M}^{\eta_2\bu N}\,&\Leftrightarrow\,\exists\ \si\in M\,,\tau\in N\ \,{\rm with}\ \,\xi\bu(\eta_1\bu\si)=\eta_2\bu\tau\\
&\Leftrightarrow\,\exists\ \si\in M\,,\tau\in N\ \,{\rm with}\ \,\big(\eta_2^{-1}\xi\,\eta_1\big)\bu\si=\tau\\
&\Leftrightarrow\,\eta_2^{-1}\xi\,\eta_1\in\widetilde\Xi_M^N\\
&\Leftrightarrow\,\xi\in\eta_2\,\widetilde\Xi_M^N\eta_1^{-1}.
\end{aligned}$$
\end{proof}

\begin{lem} \label{triviallem}
Let $M,N\subset\Si$\,. Then 
$$
{\sf Sat}(M)\cap N\ne\emptyset\Leftrightarrow{\sf Sat}(M)\cap{\sf Sat}(N)\ne\emptyset\Leftrightarrow\,\widetilde\Xi_M^N\ne\emptyset\,. 
$$
\end{lem}

\begin{proof}
One has $\,{\sf Sat}(M)\cap{\sf Sat}(N)\ne\emptyset$\, if and only if there exist $\xi_1,\xi_2\in\Xi$\,, $\si\in M$ and $\tau\in N$ such that $\xi_1\!\bu\!\si=\xi_2\!\bu\!\tau$, which is equivalent to $\big(\xi_2^{-1}\xi_1\big)\!\bu\!\si=\tau\in N$, so we have $\xi_2^{-1}\xi_1\in\widetilde\Xi_M^N$\, and $\big(\xi_2^{-1}\xi_1\big)\!\bu\!\si\in {\sf Sat}(M)$. We used the fact that
$$
\d\big(\xi_2^{-1}\big)=\r(\xi_2)=\rho(\xi_2\!\bu\!\tau)=\rho(\xi_1\!\bu\!\si)=\r(\xi_1)\,.
$$

For the converse: If $\widetilde\Xi_M^N\ne\emptyset$\,, there exists $\xi\in\Xi$ such that $\xi\bu\si=\tau$, with $\si\in M$ and $\tau\in N$\,. Then ${\sf Sat}(M)\cap N\ne\emptyset$\,, from which $\,{\sf Sat}(M)\cap{\sf Sat}(N)\ne\emptyset$ follows.
\end{proof}

\begin{rem}\label{baniciu}
Let $(\Xi,\rho,\th,\Si)$ be a continuous groupoid action and $\Delta$ a wide subgroupoid of $\Xi$\,. In terms of the restricted action from Example \ref{novotel}, if $M,N\subset\Si$\,, the contention $\widetilde{\Delta}^N_M\subset\widetilde{\Xi}^N_M$ between the corresponding recurrence sets is obvious. It is also clear that the $\Delta$-orbit of any point of $\Si$ is contained in the $\Xi$-orbit of this point and that the invariant sets under $\Xi$ are also invariant under $\Delta$\,. From this one deduces many simple connections between dynamical properties of the two actions, that we will not write down.
\end{rem}

\section{Some examples}\label{cavaler}

\subsection{Equivalence relations}\label{castandyu}

If $\Pi\subset X\!\times\!X$ is an equivalence relation on the Hausdorff topological space $X$, one can make $\Pi$ into a topological groupoid by using the product topology in $X\!\times\!X$ and the operations 
\begin{equation*}
\begin{split}
\d(x,y)=(y,y)\,,\quad &\r(x,y)=(x,x)\,,\quad (x,y)(y,z)=(x,z)\,,\quad(x,y)^{-1}\!=(y,x)\,.   
\end{split}
\end{equation*} 
The unit space is ${\sf Diag}(X)$ and we identify it with $X$, via the homeomorphism $(x,x)\mapsto x$\,.
As a particular case of Example \ref{startlet}, the groupoid $\Xi:=\Pi$ acts in a canonical way on $X$ by  
$$
(x,y)\circ y=x,\quad \forall\,x,y\in X.
$$ 
For every $M,N\subset X$ we get
\begin{equation}\label{uncazz}
\widetilde\Xi_M^N=\Xi_M^N=\Pi\cap(N\!\times\!M)\,.
\end{equation}
There are two extreme particular cases: (i) $\Pi={\sf Diag}(X)$ ({\it the trivial groupoid}), for which $\Xi_M^N$ may be identified with $N\cap M$, and (ii) $\Pi=X\!\times\! X$ ({\it the pair groupoid}), when $\Xi_M^N=N\!\times\!M$. Actually, an equivalence relation on $X$ is a wide subgroupoid of the pair groupoid.

\begin{rem}\label{tutosh}
We keep the same notations. The first projection $X\!\times\! X\to X$ is always open. The restriction to a subset ($\Pi$ in our case)  may not be an open function in general. (However, when $\Pi$ is an open subset in $X\!\times\! X$, it is also an open groupoid.) One special example \cite[Ex.\,6.2]{Wi} consists of setting $X=\mathbb R$\,, the relation being $x\,\Pi\,y \Leftrightarrow x,y\in[-1,1]\textup{ or }x=y$\,. This groupoid is Haussdorf and locally compact, but it is not open: observe that $\Xi\circ\big(\!-\tfrac{1}{2},\tfrac{1}{2}\big)=[-1,1]$\,. This is relevant for Remark \ref{caofi}, Proposition \ref{intclo} and their subsequent consequences.
\end{rem}

Even if the equivalence relation is open, very often "translations" of open sets are not open. For the pair groupoid, if $(x,y)\in X\times X$ and $y\in U\subset X$ is open, then $(x,y)\circ U=\{x\}$\,.

\smallskip
Non-open equivalence relations will be used repeatedly as counterexamples. It is true, however, that in some situations one considers on equivalence relations topologies which are different from the one induced from the Cartesian product.

\subsection{Group actions}\label{castyanu}

We indicate two ways to encode group actions by groupoids. The second one will be convenient for our purposes. We thought it would be interesting to also mention the first one, since it seems natural.

\begin{ex}\label{starlet}
As a particular case of Example \ref{startlet}, {\it  the transformation groupoid} $\Xi\equiv\G\ltimes_\gamma\!X$ associated to the continuous action $\gamma$ of the topological group $\G$ on the topological space $X$ naturally acts on $X$ by $(a,x)\circ x:=\gamma_a(x)$\,. We recall that, as a topological space, it is just $\G\times X$. The composition is $\big(b,\gamma_a(x)\big)(a,x):=(ba,x)$ and inversion reads $(a,x)^{-1}:=\big(a^{-1}\!,\gamma_a(x)\big)$\,. If $M,N\subset\Si\equiv X$, then
$$
\widetilde\Xi_M^N=\Xi_M^N=\big\{(a,x)\in\G\!\times\!M\,\big\vert\,\gamma_a(x)\in N\big\}\subset\G\!\times\!X.
$$
The first projection $p:\G\!\times\!X\to\G$ restricts to a surjection
\begin{equation}\label{debila}
p:\widetilde\Xi_M^N\to{\rm Rec}_\gamma(M,N):=\big\{a\in\G\,\big\vert\,\gamma_a(M)\cap N\ne\emptyset\big\}\subset\G,
\end{equation}
where {\it ${\rm Rec}_\gamma(M,N)$ is the usual recurrence set for dynamical systems} \cite{Au,dV,EE}. Injectivity fails in general: for instance, $\widetilde\Xi_X^X=\G\!\times\!X$ while ${\rm Rec}_\gamma(X,X)=\G$\,. On the other hand, if $M$ or $N$ are singletons, injectivity holds. Using slightly simplified notations, one has $\widetilde\Xi_{x_0}^N\cong{\rm Rec}_\gamma(x_0,N)$ and $\widetilde\Xi^{x_0}_M\cong{\rm Rec}_\gamma(M,x_0)$\,. Although the relation \eqref{debila} is quite concrete, it will not be precise enough to make suitable connections between dynamical properties in the group and in the groupoid framework. 
\end{ex}

\begin{ex}\label{bunul}
So we implement differently the classical dynamical system $(\G,\gamma,\Si)$ (for a better correspondence of notations, we set $\Si$ for the space of the group action). The group is an open groupoid in the obvious way; so we have $\Xi:=\G$ and the unit space $X=\{\e\}$ is only composed of the unit of the group. The source and the domain maps are constant, the same being true for $\rho:\Si\to\{\e\}$\,; it follows that $\G\!\Join\!\Si=\G\!\times\!\Si$\,. One sets $a\bu\si:=\gamma_a(\si)$ for every $a\in\G,\si\in\Si$ (thus $\th=\gamma$). Note that this is not covered by Examples \ref{startlet} or \ref{valtoare}. A simple inspection of the definitions shows that 
\begin{equation}\label{bunika}
\widetilde\Xi_S^T\equiv\widetilde\G_M^N={\rm Rec}_\gamma(M,N)\,,
\end{equation} 
and this will be very convenient below.
\end{ex}

\subsection{The Deaconu-Renault groupoid}\label{caterinka}

Let $\nu:X\to X$ be a local homeomorphism of the Hausdorff topological space $X$. 
To unify many constructions in groupoid $C^*$-algebras, one defines {\it the Deaconu-Renault groupoid} \cite{De,Re}
$$
\Xi(\nu):=\big\{(x,k-l,y)\,\big\vert\,x,y\in X\,,k,l\in\N\,,\nu^k(x)=\nu^l(y)\big\}\,,
$$
with structure maps
$$
(x,n,y)(y,m,z):=(x,n+m,z)\,,\quad(x,n,y)^{-1}:=(y,-n,x)\,,
$$
$$
\d(x,n,y):=y\,,\quad\r(x,n,y):=x\,.
$$
With a suitable topology (not needed here), it turns into a Hausdorff \'etale groupoid over $X$. If $\nu$ is a (global) homeomorphism, this results in the transformation groupoid of Example \ref{starlet} for the group $\G=\Z$\,.

\smallskip
An important tool is {\it the canonical cocycle} (a groupoid morphism)
$$
c:\Xi(\nu)\to\mathbb Z\,,\quad c(x,n,y):=n\,,
$$
which is, of course, the restriction to $\Xi(\nu)$ of the middle projection. For every $M,N\subset X$ set
$$
\mathbb Z_\nu(N,M):=\big\{k-l\,\big\vert\,\nu^k(N)\cap\nu^l(M)\ne\emptyset\big\}\subset\Z\,.
$$
The orbit of $y$ is $\mathcal O_y=\{x\in X\mid\Z_\nu(x,y)\ne\emptyset\}$\,, where notationally we identify singletons to points. The canonical cocycle restricts to a surjection
\begin{equation}\label{surjdr}
c:\widetilde{\Xi(\nu)}^N_M\equiv\Xi(\nu)^N_M\to\mathbb Z_\nu(N,M)\,.
\end{equation}
In contrast with the global case, if $M=\{y\}$\,, \eqref{surjdr} could still fail to be injective. But clearly $c$ restricts to a one to one map allowing to identify $\Xi(\nu)^x_y$ with $\mathbb Z_\nu(x,y)$ for every $x,y\in X$\,.

\subsection{Non-invariant restrictions}\label{creasta}

Let $(\Xi,\rho,\bu,\Si)$ be a continuous groupoid action. Let $\underline\Si$ an open subset of $\Si$\,; we do not suppose it invariant. One defines $\underline X:=\rho(\underline\Si)$\,, so $\underline\rho:=\rho|_{\underline\Si}:\underline\Si\to\underline X$ is a continuous surjection. We also have the (maybe non-invariant) groupoid restriction 
\begin{equation}\label{partz}
\underline\Xi:=\Xi_{\underline X}^{\underline X}=\{\xi\in\Xi\!\mid\!\d(\xi),\r(\xi)\in\underline X\}\,,
\end{equation}
that acts naturally by the restriction of $\bu$ on $\underline\Si$\,. Obviously 
$$
\underline\Xi\!\Join\!\underline\Si=(\Xi\!\Join\!\Si)\cap\big(\underline\Xi\!\times\!\underline\Si\big)\,.
$$ 
If $M\subset\Si$ is $\bu$-invariant, $M\cap\underline\Si$ is clearly $\underline\bu$-invariant. If $M,N\subset\Si$\,, one easily checks that 
\begin{equation}\label{curestrictie}
\widetilde{\underline\Xi}^{\,N\cap\underline\Si}_{\,M\cap\underline\Si}\subset\underline\Xi\cap\widetilde\Xi^N_M=\widetilde{\underline\Xi}^{\,N\cap\,\rho^{-1}[\rho(\underline\Si)]}_{\,M\cap\,\rho^{-1}[\rho(\underline\Si)]}\,.
\end{equation}
One gets equality whenever $\underline\Si$ is {\it $\rho$-saturated}, i.e.\! when $\rho^{-1}\big[\rho(\underline\Si)\big]=\underline\Si$\,, in particular when $\rho$ is a bijection. On the other hand, if $M,N\subset\underline\Si$\,, one always gets equality in \eqref{curestrictie}, and in this case
\begin{equation}\label{cuprecista}
\widetilde{\underline\Xi}^{\,N}_{\,M}=\underline\Xi\cap\widetilde\Xi^N_M=\widetilde\Xi^N_M\,.
\end{equation}
To get the last term, one notices that $\widetilde\Xi^N_M\subset\Xi^{\rho(N)}_{\rho(M)}\subset\underline\Xi$\,, since $\rho(M),\rho(N)\subset\underline X$\,, or proceeds directly.

\subsection{Partial group actions}\label{castanika}

Let $\G$ be a group assumed, for simplicity, to be discrete. The unit is denoted by $\e$\,. 

\begin{defn}\label{iexel}
{\it A partial action} \cite{Ex1} of $\G$ on the topological space $Y$ is a family of homeomorphisms
$$
\beta:=\Big\{Y_{a^{-1}}\!\overset{\beta_a}{\longrightarrow}Y_a\,\Big\vert\,a\in\G\Big\}
$$
between open subsets of $Y$ satisfying 
\begin{enumerate}
\item[(i)] $Y_\e=Y$ and $\beta_\e={\rm id}_X$\,,
\item[(ii)] $\beta_a\circ\beta_b$ is a restriction of $\beta_{ab}$ for every $a,b\in\G$\,.
\end{enumerate}
\end{defn}
The domain of the composition above is 
$$
{\rm dom}(\beta_a\circ\beta_b):=\big\{y\in Y_{b^{-1}}\,\big\vert\,\beta_b(y)\in Y_{a^{-1}}\big\}=\beta_b^{-1}\big(Y_{a^{-1}}\big)\,.
$$ 
It is easily shown that $\beta_a^{-1}\!=\beta_{a^{-1}}$ for any element $a$\,. To such a partial action, one associates \cite{Ab} an \'etale groupoid $\Xi[\beta]\equiv\G\ltimes_{[\beta]}\!Y$ over $Y$. As a set it is 
$$
\G\ltimes_{[\beta]}\!Y:=\{(a,y)\,\big\vert\,a\in\G,y\in Y_{a^{-1}}\}\,,
$$
on which one considers the induced product topology. The structure maps are
$$
\d(a,y):=y\,,\quad\r(a,y):=\beta_a(y)\,,
$$
$$
\big(a,\beta_b(y)\big)(b,y):=(ab,y)\,,\quad(a,y)^{-1}\!:=\big(a^{-1},\beta_a(y)\big)\,.
$$
If the action is global, meaning that $Y_a=Y$ for every $a$\,, then we recover the situation of Example \ref{starlet}.

\smallskip
The subset $T\subset Y$ is said to be {\it $\beta$-invariant} if $\beta_a(y)\in T$ whenever $a\in\G$ and $y\in T\cap Y_{a^{-1}}$. Of course, this happens exactly when $T$ is invariant, seen as a set of units of $\Xi[\beta]$\,.

\smallskip
For $S,T\subset Y$ it seems natural to define {\it the recurrence set}
\begin{equation}\label{vomitiv}
{\rm Rec}_{[\beta]}(S,T):=\big\{a\in\G\,\big\vert\,\exists\,y\in Y_{a^{-1}}\cap S\,,\,\beta_a(y)\in T\big\}\,.
\end{equation}
Very much as in Example \ref{starlet}, one gets
$$
\widetilde{\Xi[\beta]}_S^T=\Xi[\beta]_S^T=\big\{(a,y)\in\G\!\times\!S\,\big\vert\,y\in Y_{a^{-1}},\,\beta_a(y)\in T\big\}\subset\G\!\ltimes_{[\beta]}\!Y.
$$
The first projection $(a,y)\to p(a,y)\!:=a$ restricts to a surjection
\begin{equation}\label{debilla}
p:\widetilde{\Xi[\beta]}_S^T\to{\rm Rec}_{[\beta]}(S,T)\,.
\end{equation}
If $S=\{y_0\}$ is a singleton, injectivity holds and one has $\widetilde{\Xi[\beta]}_{y_0}^T\cong{\rm Rec}_{[\beta]}(y_0,T)$\,.

\smallskip
A very direct way to get a partial group action is to make a non-invariant restriction in a group (global) action. In a certain sense this is the most general situation, since any partial action can be extended to a global one \cite{Ex}. 

\smallskip
So let $(\G,\gamma,X)$ be a continuous group action and $Y$ an open, maybe non-invariant, subset of $X$. For every $a\in\G$, set $Y_{a}:=Y\cap\gamma_a(Y)\subset Y$. It is straightforward to check that
$$
\beta:=\big\{\beta_a:=\gamma_a|_{Y_{a^{-1}}}\!:Y_{a^{-1}}\!\to Y_a\,\big\vert\,a\in\G\big\}
$$ 
becomes a partial action. The groupoid $\Xi[\beta]\equiv\G\ltimes_{[\beta]}\!Y$ may be seen as the non-invariant restriction to $Y$ of the transformation groupoid $\G\ltimes_\gamma\!X$ from Example \ref{starlet}, as described in Subsection \ref{creasta}. For this notice that, if $(a,x)\in\G\!\times\!X$, the conditions $\d(a,x),\r(a,x)\in Y$ mean exactly that $x\in Y_{a^{-1}}$\,. The relationship between the two types of groupoid recurrence sets may be read off from \eqref{curestrictie}, or from \eqref{cuprecista} in the most interesting case, when the two sets are already contained in $Y$. In this last case, one gets
\begin{equation}\label{nickcave}
\big(\G\ltimes_{[\beta]}\!X\big)_S^T=\big(\G\ltimes_\gamma\!X\big)^T_S\,,
\end{equation}
which can also be checked directly. In terms of the actions themselves, using a notation from \eqref{debila}, the relation \eqref{vomitiv} converts into
\begin{equation}\label{curestriczie}
{\rm Rec}_{[\beta]}(S,T)=\big\{a\in\G\,\big\vert\,\exists\,y\in \gamma_{a^{-1}}(Y)\cap S\,,\,\gamma_a(y)\in T\big\}={\rm Rec}_\gamma(S,T)\,,
\end{equation}
since $\gamma_a(y)\in T\subset Y$ already implies that $y\in\gamma_{a^{-1}}(Y)$\,. It is also easy to verify that \eqref{curestriczie} follows from \eqref{debilla}, \eqref{nickcave} and \eqref{debila}.

\subsection{The action groupoid of a groupoid action}\label{castanier}

A continuous groupoid action $(\Xi,\rho,\bu,\Si)$ being given, one constructs \textit{the action (or transformation, or crossed product) groupoid} which, as a set, is the closed subspace $\Xi\!\Join\!\Si$ of $\Xi\!\times\!\Si$ introduced in \eqref{ganchor} and the structure maps are
\begin{equation*}\label{gaucho}
\big(\eta,\xi\bu\si)(\xi,\si):=(\eta\,\xi,\si)\,,\quad(\xi,\si)^{-1}\!:=\big(\xi^{-1}\!,\xi\bu\si\big)\,,
\end{equation*}
\begin{equation}\label{garucho}
{\sf d}(\xi,\si):=(\rho(\si),\si)\equiv\si\,,\quad{\sf r}(\xi,\si):=({\rm r}(\xi),\xi\bu\si)\equiv\xi\bu\si.
\end{equation}
To stress the origin of the construction, we are going to denote by $\Xi\!\ltimes_\th\!\Si$ this groupoid, in analogy with the group case which is a particular example.
The space of units $(\Xi\!\ltimes_\th\!\Si)^{(0)}$ identifies with $\Si$\,. The canonical action of $\Xi\!\ltimes_\th\!\Si$ on $\Si$ (see Example \ref{startlet})
\begin{equation*}\label{tifies}
(\xi,\si)\circ(\rho(\si),\si)=(\xi,\si)(\rho(\si),\si)(\xi,\si)^{-1}=(\r(\xi),\xi\bu\si)=(\rho(\xi\bu\si),\xi\bu\si)
\end{equation*}
may also be written $(\xi,\si)\circ\si=\xi\bu\si$, and thus reproduces in some way the initial action. The invariant subsets, in particular the orbits, are the same. This has a series of obvious consequences on the way dynamical properties are preserved when passing from $(\Xi,\rho,\bu,\Si)$ to $\Xi\!\ltimes_\th\!\Si$\,. 

\smallskip
The connection between the relevant recurrent sets is similar with (and in fact generalizes) that of Example \ref{starlet}: if $M,N\subset\Si$ we get
\begin{equation*}\label{lafel}
(\Xi\!\ltimes_\th\!\Si)^N_M=\{(\xi,\si)\in\Xi\!\ltimes_\th\!\Si\!\mid\!\si\in M\,,\,\xi\bu\si\in N\}\,.
\end{equation*}
Consequently
\begin{equation}\label{consequently}
\widetilde\Xi^N_M=p\big[(\Xi\!\ltimes_\th\!\Si)^N_M\big]\,,
\end{equation}
where $p$ is the restriction of the first projection of the product $\Xi\!\times\!\Si$\,. This restriction is not always injective, being constant on any set of the type $\{\xi\}\!\times\!\rho^{-1}[\d(\xi)]$\,. It follows that \textit{for recurrence phenomena it is not always a good idea to replace the initial action by the action groupoid}.

\subsection{Groupoid pull-backs}\label{cabalier}

There is a powerful method to construct new more sophisticated groupoids from simpler ones, which, however, does not loose control over the orbit structure or the recurrence sets. Let ${\rm d,r}:\Xi\to X$ be the domain and the range maps of a topological groupoid, $\O$ a topological space and $h:\O\to X$ an open continuous surjection. Let
$$
h^{\downarrow\downarrow}(\Xi):=\{(\o,\xi,\o')\in \O\times\Xi\times \O\mid{\rm r}(\xi)=h(\o),{\rm d}(\xi)=h(\o')\}
$$ 
be {\it the associated pull-back groupoid} \cite{CNQ,Mac}. We recall its structural maps:
$$
(\o_1,\xi,\o_2)(\o_2,\eta,\o_3):=(\o_1,\xi\eta,\o_3)\,,\quad(\o,\xi,\o')^{-1}\!:=\big(\o',\xi^{-1}\!,\o\big)\,,
$$
$$
{\mathfrak d}(\o,\xi,\o'):=\o'\,,\quad{\mathfrak r}(\o,\xi,\o'):=\o\,.
$$
Note the relations between orbits and orbit closures in the two groupoids (closures commute with open continuous surjections):
\begin{equation*}\label{relorb}
\mathcal O^{\downarrow\downarrow}_{\o}=h^{-1}\big(\mathcal O_{h(\o)}\big)\,,\quad \overline{\mathcal O^{\downarrow\downarrow}_{\o}}=h^{-1}\big(\overline{\mathcal O_{h(\o)}}\big)\,.
\end{equation*}
For $M,N\subset\O$\,, by inspection one gets
\begin{equation}\label{transfers}
h^{\downarrow\downarrow}(\Xi)^N_M=N\!\times\!\Xi_{h(M)}^{h(N)}\!\times\! M.
\end{equation}

\section{Topological transitivity}\label{calorifiert}

\subsection{The standard notions}\label{motan}

Let us fix a continuous groupoid action $(\Xi,\rho,\th,\Si)$\,. We start with the simplest notions. 
If there is just one orbit, the action is {\it transitive}. This happens for the pair groupoid, for instance. (In \cite[Prop.\,3.18]{EK} it is shown that a compact transitive groupoid is open.) A point having a dense orbit is called {\it a transitive point}. If there is a dense orbit, i.e.\! if a transitive point does exist, one says that the action is {\it pointwise transitive}. 
For the more refined notions, one needs first to prove the next result:

\begin{thm}\label{pricinoasa}
Let us consider the following conditions:
\begin{enumerate}
\item[(i)]
$\Si$ is not the union of two proper invariant closed subsets.
\item[(i')]
Any two open non-void invariant subsets of $\,\Si$ have non-trivial intersection.
\item[(ii)]
Each non-empty open invariant subset of $\,\Si$ is dense.
\item[(iii)]
For every $\,U,V\subset\Si$ open and non-void,  $\widetilde\Xi_U^V\ne\emptyset$ holds ({\rm recurrent transitivity}).
\item[(iv)]
Each invariant subset of $\,\Si$ is ether dense, or nowhere dense ({\rm topological transitivity}).
\end{enumerate}
Then the following implications hold:
\begin{enumerate}
\item
Transitivity $\,\Rightarrow$ pointwise transitivity $\,\Rightarrow$ recurrent transitivity (iii). 
\item
One has $(iv)\Rightarrow (iii)\Rightarrow (ii)\Rightarrow (i')\Leftrightarrow (i)$\,. None of the other implications holds in general. Pointwise transitivity does not imply topological transitivity.
\item
If the groupoid $\Xi$ is open, the conditions $(i)$ to $(iv)$ are all equivalent. 
\end{enumerate}
\end{thm}

\begin{proof}
1. The first implication is trivial (and obviously it is not an equivalence). We verify now the second one.  Assume that $\Si$ has a dense orbit $\mathfrak O_\si$ and let $\emptyset\ne V_1,V_2$ be open sets. One has 
$$
\xi_1\bu\si\in\mathfrak O_\si\cap V_1\ne\emptyset\quad {\rm and}\quad \xi_2\bu\si\in\mathfrak O_\si\cap V_2\ne\emptyset
$$ 
for elements $\xi_1,\xi_2\in\Xi$ with $\r\big(\xi_1^{-1}\big)=\d(\xi_1)=\rho(\si)=\d(\xi_2)$\,. Since $\big(\xi_2\,\xi_1^{-1}\big)(\xi_1\bu\si)=\xi_2\bu\si$ we infer that $\xi_2\,\xi_1^{-1}\!\in\widetilde\Xi_{V_1}^{V_2}\ne\emptyset$\,.

\smallskip
2. $(i)\Leftrightarrow(i')$ If (i) fails, i.\,e. $\Sigma=C\cup D$ with $C$ and $D$ proper closed invariant subsets, then $C^c\cap D^c=\emptyset$\,. This contradicts $(i')$, since $C^c$ and $D^c$ are open, non-void and invariant. On the other hand, if $A,B$ are open non-empty invariant sets such that $A\cap B=\emptyset$\,, then $A^c\cup B^c=\Si$ with $A^c,B^c$ proper closed invariant subsets, finishing the proof of the equivalence.

\smallskip
$(ii)\Rightarrow(i')$ If each non-void open invariant subset is dense, it meets every other (invariant) non-void open set.

\smallskip
$(iii)\Rightarrow(ii)$ Let $\emptyset\ne U\subset\Si$ open and invariant. By assumption, for every non-void open set $V\subset\Si$ there exists some $\xi\in\Xi$ making $(\xi\bu U)\cap V$ non-void. But $\xi\bu U\subset U$, implying $U\cap V\ne\emptyset$\,. Thus, $U$ meets every other non-void open set and must be dense.

\smallskip
$(iv)\Rightarrow(iii)$ Suppose $(iv)$ holds. Let $\emptyset\ne U,V\subset\Si$ open sets. ${\sf Sat}(U)$ is an invariant set containing $U$, so it cannot be nowhere dense, meaning that it is dense. Hence ${\sf Sat}(U)\cap V\ne\emptyset$ and we conclude by Lemma \ref{triviallem}.

\smallskip
The fact that all the other implications fail without extra assumptions will be showed in a series of counterexamples below.

\smallskip
3. Provided that $\d$ is open, it is enough to prove that $(i')$ implies $(iv)$\,.  So let us assume $(i')$\,, but let $A\subset\Si$ be invariant, neither dense, nor nowhere dense. Then $\big(\overline{A}\,\big)^\circ$ and $(A^c)^\circ=\big(\,\overline A\,\big)^c$ are both non-void open sets, which are invariant by Proposition \ref{intclo}. They should meet, by $(i')$\,, but this is obviously false.
\end{proof}

Let us indicate now the necessary counterexamples for groupoids that are not open.

\begin{ex}\label{valioso}
This example shows that $(ii)\not\Rightarrow(iii),(iv)$ (redundantly but explicitly). Let 
$$
\Si=X\equiv\mathbb R^{\sf x}\!:=(-\infty,0)\cup(0,\infty).
$$ 
Define the sets 
$$
X_1=\{-1\}\cup\big((0,\infty)\setminus\{1\}\big)\quad\textup{ and }\quad X_2=\big((-\infty,0)\setminus\{-1\}\big)\cup \{1\}\,.
$$ 
Forming a partition of $\mathbb R^{\sf x}$, they induce an equivalence relation $\Pi$\,, viewed as a topological groupoid acting on $\mathbb R^{\sf x}$ as in subsection \ref{castandyu}. It is easy to check that $\Pi$ is not open. The orbits are precisely $X_1$ and $X_2$ (neither open nor closed), with closures 
\begin{equation*}\label{sughit}
\overline X_1=\{-1\}\cup (0,\infty)\quad{\rm and}\quad\overline X_2=(-\infty,0)\cup\{1\}\,.
\end{equation*}
These orbit closures are not invariant, showing already that for non-open groupoids the conclusion of Proposition \ref{intclo} does not necessarily hold. (And we see that $\Pi$ is not an open groupoid in an indirect way.)

\smallskip
It is easy to see that $X=X_1\cup X_2$ is the single (non-void) open invariant set, since any invariant set is a union of orbits, and $X_1,X_2$ are not open. Thus the condition $(ii)$ from Theorem \ref{pricinoasa} holds. On the other hand the condition $(iv)$ definitely fails, since the invariant sets $X_1,X_2$ are neither dense, nor nowhere dense. Of course recurrent transitivity also fails. To see this directly, take for instance $U=(2,\infty)$ and $V=(-\infty,-2)$\,, for which 
$$
\widetilde\Xi_U^V=\{(x,y)\in\Pi\mid y\in U,x\in V\}=\emptyset\,,
$$ 
by \eqref{uncazz} or by a simple computation. 
\end{ex}

\begin{ex}\label{valioso2} The previous example can be easily modified to show that $(i')\not\Rightarrow (ii)$. In $\mathbb{R}^{\sf x}$, define the equivalence relation $\Upsilon$ associated to the partition 
$$
Y_1=\{-1\}\cup(0,1)\cup(1,2)\cup(3,\infty)\,,\quad X_2=\big((-\infty,0)\setminus\{-1\}\big)\cup \{1\}\,,\quad\,Y_3=[2,3]
$$ 
and consider the canonical action of $\,\Xi=\Upsilon$ on $\mathbb{R}^{\sf x}$. As invariant sets are union of orbits, the only non-void invariant open sets are $\mathbb{R}^{\sf x}$ and 
$$
Y_1\cup X_2=(-\infty,0)\cup(0,2)\cup(3,\infty)\,,
$$ 
which have non-empty intersection, so $(i')$ holds. But the set $Y_1\cup X_2$ is not dense, hence $(ii)$ fails.
\end{ex}

\begin{ex}\label{valioso3} Define on $\Si=X=\mathbb R$ the equivalence relation 
$$
x\,\Pi\,y\ \Leftrightarrow\ x,y\in\mathbb Q\ \ \textup{or}\ \ x=y\,.
$$ 
This provides a non-open groupoid acting on $\mathbb R$\,, and this action is pointwise transitive ($\mathbb Q$ is a dense orbit), hence recurrently transitive. The invariant set $(0,1)\setminus \mathbb Q$ is neither nowhere dense nor dense, so this action is not topologically transitive. This example shows that {\it pointwise  transitivity} $\not\Rightarrow(iv)$ and (consequently) $(iii)\not\Rightarrow(iv)$.
\end{ex}

\begin{prop}\label{uncorolar}
If $\,\Xi$ is open and $\,\Si$ is a Baire second-countable space, then topological transitivity, recurrent transitivity and pointwise transitivity are equivalent (and also equivalent to the properties $(i)$ and $(ii)$).
\end{prop}

\begin{proof}
Having in view Theorem \ref{pricinoasa}, what remains is to show that $(ii)$ implies pointwise transitivity. Since $\Sigma$ is second-countable, its topology has a countable basis $\{V_n\ne\emptyset\}_{n\in\mathbb N}$\,. By defining $U_n=\Xi\bu V_n$ we get countably many dense open subsets of a Baire space, so $U=\cap_n U_n$ is also a dense (invariant) set. Let $W\ne\emptyset$ be an open subset of $\Si$\,. By the definition of a basis, there exists some $V_n\subset W$. Hence we have $U\subset U_n=\Xi\bu V_n\subset\Xi\bu W$. Therefore, if $\si\in U$ then $\xi^{-1}\!\bu\si\in W$ for some $\xi\in\Xi$\,.
Hence $\si$ has a dense orbit and the action is pointwise transitive. 
\end{proof}

We recall that Hausdorff locally compact spaces and complete metric spaces are Baire.

\begin{rem}\label{diversiune}
 What we used in the proof of Proposition \ref{uncorolar} is the fact that the intersection $U$ is non-void. Actually, one could improve: {\it under the given requirements, the set of points with dense orbit is a dense $G_\delta$-set.}
\end{rem}

\begin{ex}\label{despregrup}
In the classical dynamical system case of Example \ref{bunul} we recover known results \cite{Au,dV,Gl}. Since the source map $\G\to\{\e\}$ is clearly open, Theorem \ref{pricinoasa} simplifies a lot. Even in this particular case, without second countability the full equivalence from Proposition \ref{uncorolar} fails.
\end{ex}

\begin{ex}\label{ceacpac}
The Deaconu-Renault groupoid is recurrently transitive if and only if for every $\emptyset\ne U,V\subset X$ open sets, there exist $x\in V,\,y\in U,\,k,l\in\N$ such that $\nu^k(x)=\nu^l(y)$\,. 
\end{ex}

\begin{ex}\label{fiarra}
We say that the partial group action $(\G,\beta,Y)$ from Subsection \ref{castanika} is {\it recurrently transitive} if ${\rm Rec}_{[\beta]}(U,V)\ne\emptyset$ for every $U,V\subset Y$ non-void and open. From \eqref{debilla} one deduces that this happens precisely when the groupoid $\G\!\ltimes_{[\beta]}\!Y$ is recurrently transitive.
\end{ex}

\begin{ex}\label{hupatup}
As in Subsection \ref{creasta}, let $(\Xi,\rho,\th,\Si)$ be a groupoid action and let $\big(\underline\Xi,\underline\rho,\underline\th,\underline\Si\big)$ be the action obtained by restriction to the open, maybe non-invariant subset $\underline\Si\subset\Si$\,. {\it If  $\th$ is recurrently transitive, then $\underline\th$ is also recurrently transitive.} This follows from formula \eqref{cuprecista} and from the fact that open subsets of $\underline\Si$ are open in $\Si$\,. It is clear that in general there is no converse to the statement. In particular, $\underline\Si$ might happen to be invariant and recurrently transitive, while the restriction to $\Si\!\setminus\!\underline\Si$ is not constrained in any way. The next proposition outlines a favorable situation.
\end{ex}

\begin{prop}\label{constrained}
Suppose  that $\Xi$ is open and ${\sf Sat}\big(\underline\Si\big)$ is dense in $\Si$\,. Then $\th$ is recurrently transitive if and only if $\underline\th$ is recurrently transitive.
\end{prop}

\begin{proof}
Suppose that $\underline\th$ is recurrently transitive and let $U,V$ be two open non-void subsets of $\Si$\,. By density, $U$ and $V$ both intersect the saturation $\Xi\bu\underline\Si$\,. So Lemma \ref{triviallem} tells us that the open set $\underline\Si$\, intersects the open sets $\Xi\bu U$ and $\Xi\bu V$. Use the fact that $\underline\th$ is recurrently transitive and Lemma \ref{triviallem} (again) to conclude that $\widetilde{\Xi}^V_U\ne\emptyset$\,.
\end{proof}

\begin{cor}\label{scurtatura}
Let $(\G,\gamma,X)$ be a global group action and $(\G,\beta,Y)$ the partial action obtained as in Subsection \ref{castanika} by restricting to the open subset $Y$. We assume that $X$ is the smallest invariant set containing $Y$. Then one of the actions is recurrently transitive if and only if the other is so.
\end{cor}

In \cite[Prop.\,5.5]{Ex1} it is shown that any partial action $(\G,\beta,Y)$ may be obtained from a global one, with the extra condition that the total space $X$ is the saturation of the initial one, in an essentially unique way. In general $X$ could fail to be Hausdorff, cf. \cite[Prop.\,5.6]{Ex1}. The results are attributed to F. Abadie \cite{Ab}.

\begin{ex}\label{despreswitch}
In the setting of subsection \ref{castanier},  {\it the groupoid action $(\Xi,\rho,\bu,\Si)$ is recurrently transitive if and only if the action groupoid $\Xi\!\ltimes_\th\!\Si$ is so} ($p$ is surjective, hence $p(B)$ is void if and only if $B$ is void). 
\end{ex}

\begin{ex}\label{desprepulbec}
Although the pullback groupoid of subsection \ref{cabalier} might be much more complicated than the initial one, $h^{\downarrow\downarrow}(\Xi)$ and $\Xi$ are simultaneously recurrently transitive. This follows from \eqref{transfers}.
\end{ex}

\subsection{Weak pointwise transitivity}\label{pisica}

\begin{defn}
Define {\it the invariant closure of} $A\subset\Si$ by 
$$
\mathfrak C(A):=\!\!\bigcap\limits_{\substack{A\subset M \\ M\textup{ closed, invariant}}}\!\!\!M\,.
$$ 
This is, the smallest closed and invariant set including $A$\,. One very special case (deserving its own notation) is $\mathfrak C_\si=\mathfrak C(\{\si\})$\,, {\it the invariant orbit closure of $\si$}.
\end{defn}

\begin{prop} The invariant closure verifies 
$$
\mathfrak C(A)=\mathfrak C({\sf Sat}(A))=\mathfrak C\big(\,\overline{{\sf Sat}(A)}\,\big)\,.
$$ 
In particular 
$$
\mathfrak C_\si=\mathfrak C(\mathfrak O_\si)=\mathfrak C(\overline{\mathfrak O}_\si)\,\supset\overline{\mathfrak O}_\si\,.
$$ 
If $\,\Xi$ is open, $\mathfrak C(A)=\overline{{\sf Sat}(A)}$ and $\mathfrak C_\si=\overline{\mathfrak O}_\si$\,.
\end{prop}

\begin{proof} 
Since $A\subset{\sf Sat}(A)\subset\overline{{\sf Sat}(A)}$\,, one readily gets 
$$
\mathfrak C(A)\subset\mathfrak C({\sf Sat}(A))\subset\mathfrak C\big(\,\overline{{\sf Sat}(A)}\,\big)\,.
$$
The set $\mathfrak C(A)$ is closed and invariant, so it contains $\overline{{\sf Sat}(A)}$\,, implying that $\mathfrak C\big(\,\overline{{\sf Sat}(A)}\,\big)\subset\mathfrak C(A)$\,. 

\smallskip
If $\Xi$ is open, then by Proposition \ref{intclo}, $\overline{{\sf Sat}(A)}$ is a closed and invariant set containing $A$\,, so $\mathfrak C(A)=\overline{{\sf Sat}(A)}$\,. Then use the fact that $\mathfrak O_\si={\sf Sat}(\{\si\})$\,.
\end{proof}

It is not difficult to see that $\mathfrak C(A)$ contains the sets 
$$
\overline{{\sf Sat}(A)}\,,\overline{{\sf Sat}(\overline{{\sf Sat}(A)})}\,, \overline{{\sf Sat}(\overline{{\sf Sat}(\overline{{\sf Sat}(A)})})}\,,\ldots
$$ 
Whenever the groupoid is open, this ascending chain of sets collapses in the first step. The relevance of the set $\mathfrak C(A)$ in groupoids that are not open is a new phenomenon. 
 
\begin{defn}\label{wpt} An action $(\Xi,\rho,\theta,\Si)$ is called {\it weakly pointwise transitive} $(wpt)$ if exists a point $\si\in\Si$ such that $\mathfrak C_\si\!=\Si$\,. In this case, we say that $\si$ is a {\it weakly transitive point}.
\end{defn}

\begin{rem}
Since $\overline{\mathfrak O}_\si\subset\mathfrak C_\si$\,, it is clear that {\it pointwise transitive} $\Rightarrow$ {\it weakly  pointwise transitive}. When $\Xi$ is an open groupoid, {\it weakly pointwise transitive} $\Leftrightarrow$ {\it pointwise transitive}.
\end{rem}

\begin{ex}
We compute the invariant closures in the previous examples: 
\begin{itemize}
 \item In example $\ref{valioso}$, one has $\mathfrak C_\si=\R^{\sf x}$ for all $\si\in \Si=\R^{\sf x}$. So  {\it every point is weakly transitive}. We recall that this example is not topologically transitive or recurrently transitive.
\item In example $\ref{valioso2}$, 
$$
\mathfrak C_\si=\left\{\begin{array}{ll}
[2,3]     & \textup{if\ } \si\in[2,3]\,, \\
\,\,\,\R^{\sf x}      & \textup{if\ } \si\in\R^{\sf x}\setminus[2,3]\,. \\
\end{array}\right.
$$ 
\item In example $\ref{valioso3}$, 
$$
\mathfrak C_\si=\left\{\begin{array}{ll}
\,\,\R     & \textup{if\ } \si\in\mathbb Q\,, \\
\{\si\}      & \textup{if\ } \si\not\in\mathbb Q\,. \\
\end{array}\right.
$$
\end{itemize} 

All of these are {\it weakly pointwise transitive} systems. This examples also show that the sets $\{\mathfrak C_\si\!\mid \!\si\in\Si\,\}$ doesn't need to be disjoint: some could intersect or even (strictly) contain others. Of course, this often happens for orbit closures, but for invariant closures the overlaps tend to be larger.
\end{ex}

\begin{prop}
{\it Weak pointwise transitivity} implies the condition $(i)$ of Theorem \ref{pricinoasa}.
\end{prop}

\begin{proof}
Suppose that $\Si=N_1\cup N_2$\,, being $N_1,N_2$ closed and invariant sets. Without loss of generality, there exists $\si\in N_1$ such that $\mathfrak C_\si\!=\Si$\,. As $N_1$ is closed, invariant and contains $\si$, the relation $\Si=\mathfrak C_\si\!\subset N_1$ follows. We conclude that $(i)$ holds.
\end{proof}

\begin{rem}
Example \ref{valioso2} shows that {\it weakly pointwise transitivity does not imply the condition $(ii)$} of Theorem \ref{pricinoasa}, in general. We recall that pointwise transitivity does imply $(iii)$\,, which is stronger than $(ii)$, so none of these properties is implied by weak topological transitivity. We summarize most of the implications in the following diagram, in which the arrows indicate implications:
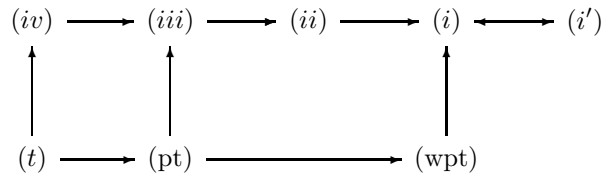
\begin{equation*}\label{coliflor}
\begin{diagram}
\node{(iv)}\arrow{e}\node{(iii)}\arrow{e}\arrow{e}\node{(ii)}\arrow{e}\node{(i)}\arrow{e,<>}\node{(i')}\\ 
\node{(t)}\arrow{n}\arrow{e}\node{(\rm pt)}\arrow{n}\arrow[2]{e}\node{}\node{(\rm wpt)}\arrow{n}
\end{diagram}
\end{equation*}
\end{rem}

\section{Recurrence of points}\label{cracifloc}

\subsection{Limit sets}\label{craciflor}

\textit{A continuous action $(\Xi,\rho,\th,\Si)$ will be fixed, with $\Xi$ strongly non-compact}. We recall from the Introduction that this means that none of the $\d$-fibers of the groupoid is compact. Of course, the topological spaces $\Si$ and $X=\Xi^{(0)}$ are allowed (but not required) to be compact. Note that closed equivalence relations on compact spaces $X$, with the induced topology, are excluded.
Let $\mathcal K(T)$ denote the family of compact subsets of the topological space $T$. 

\begin{defn}\label{pescuit}
\textit{The limit set of the point} $\si\in\Si$ is the closed subset of $\Si$
\begin{equation*}\label{nicetry}
\mathfrak L_\si^\th\equiv\mathfrak L_\si\!:=\!\bigcap_{{\sf K}\in\mathcal K(\Xi)}\!\overline{(\Xi\!\setminus\!{\sf K})\bu\si}\,.
\end{equation*}
\end{defn}

Note that 
$$
(\Xi\!\setminus\!{\sf K})\bu\si=\big[(\Xi\!\setminus\!{\sf K})\cap\Xi_{\rho(\si)}\big]\bu\si=\big[\Xi_{\rho(\si)}\!\setminus\!{\sf K}_{\rho(\si)}\big]\bu\si\,,
$$ 
where ${\sf K}_{\rho(\si)}:={\sf K}\cap\Xi_{\rho(\si)}$ is compact in $\Xi_{\rho(\si)}$\,. It follows that one can also write 
$$
\mathfrak L_\si=\!\bigcap_{{\sf k}\in\mathcal K(\Xi_{\rho(\si)})}\!\overline{\big(\Xi_{\rho(\si)}\!\setminus{\sf k}\big)\bu\si}\,.
$$
The limit set would be void if $\Xi_{\rho(\si)}$ was allowed to be compact (which is not), but it can also be void in other situations.

\smallskip
For any unit $x$ we say that the net $(\xi_i)_{i\in I}\subset\Xi_x$ \textit{diverges} if for every compact ${\sf k}\subset\Xi_x$ there exists $i_{\sf k}\in I$ such that $\xi_i\notin{\sf k}$ if $i\ge i_{\sf k}$\,. The existence of divergent nets relies on our strongly non-compact assumption.

\begin{lem}\label{labil}
The following statements for $\si,\tau\in\Si$\, are equivalent:
\begin{enumerate}
\item[(i)]
$\tau$ belongs to the limit set $\mathfrak L_\si$\,.
\item[(ii)]
For every neighborhood $V$ of $\tau$ there exists a divergent net $(\xi_i)_{i\in I}$ in $\Xi_{\rho(\si)}$ such that $\xi_i\bu\si\in V$ for any $i\in I$.
\item[(ii')]
For every neighborhood $V$ of $\tau$, the recurrence set $\widetilde\Xi_\si^V$ is not relatively compact.
\item[(iii)]
There is a divergent net $(\eta_i)_{i\in I}$ in $\Xi_{\rho(\si)}$ such that $\eta_i\bu\si\to\tau$. 
\end{enumerate}
\end{lem}

\begin{proof}
$(i)\Rightarrow(ii)$ Let $V$ be a neighborhood of $\tau$. If $\tau$ belongs to the limit set $\mathfrak L_\si$\,, then $\tau\in \overline{\big(\Xi_{\rho(\si)}\!\setminus{\sf k}\big)\bu\si}$ for every ${\sf k}\in\mathcal K(\Xi_{\rho(\si)})$\,. Hence we can choose elements $\xi_{\sf k}\in\Xi_{\rho(\si)}\!\setminus{\sf k}$ such that $\xi_{\sf k}\bu\si\in V$. The net $(\xi_{\sf k})_{{\sf k}\in\mathcal K(\Xi_{\rho(\si)})}$ is obviously divergent.

\smallskip
$(ii)\Rightarrow(iii)$ Consider the set $\mathfrak N_\tau$ of neighborhoods of $\tau$,  and order it by reversing the inclusions. For each neighborhood $V$ of $\tau$, select some divergent net $(\xi_{i,V})_{i\in I}$ such that $\xi_{i,V}\bu\si\in V$ (it can be built over the same labels). Observe that, for each ${\sf k}\in\mathcal K(\Xi_{\rho(\si)})$\,, there exists $i_{\sf k}^V$ such that $\xi_{i,V}\not\in{\sf k}$ for every $i\geq i_{\sf k}^V$. Define $\eta_{{\sf k},V}=\xi_{i_{\sf k}^V\!,V}$\,, which forms a net when $\mathfrak N_\tau\!\times\!\mathcal K(\Xi_{\rho(\si)})$ is given the product order. By construction, we get a divergent net and $\tau=\underset{{\sf k},V}{\lim}\,\eta_{{\sf k},V}\bu\si$.

\smallskip
$(iii)\Rightarrow(i)$ Let ${\sf k}\in\mathcal K\big(\Xi_{\rho(\si)}\big)$. Since $\eta_i$ is divergent, there exists $i_{\sf k}\in I$ such that $\eta_i\in\Xi_{\rho(\si)}\!\setminus{\sf k} $\,, $\forall\,i\geq i_{\sf k}$\,. So we have that 
$$
\tau=\lim_{i\in I} \eta_i\bu\si=\lim_{i\geq i_{\sf k}}\eta_i\bu\si\in \overline{\big(\Xi_{\rho(\si)}\!\setminus{\sf k}\big)\bu\si}\,.
$$ 
It follows that $\tau\in \mathfrak L_\si$\,.

\smallskip
$(ii)\Leftrightarrow(ii')$ follows from the definitions.
\end{proof}

The next easy lemma is sometimes useful to compute limit sets.

\begin{lem}\label{technicallemma}
Suppose that there exists a family of compact sets $\{{\sf k}_\lambda\}_{\lambda\in\Lambda}$ that exhausts $\Xi_{\rho(\si)}$\,. That is, $\Lambda$ is a directed set, $\Xi_{\rho(\si)}=\bigcup_{\lambda\in\Lambda}{\sf k}_\lambda$ and  ${\sf k}_{\lambda_1}\subset{\sf k}_{\lambda_2}^\circ$ whenever $\lambda_1\leq\lambda_2$\,. Then the limit set $\mathfrak L_\si$ can be computed as 
$$
\mathfrak L_\si=\bigcap_{\lambda\in\Lambda}\overline{(\Xi_{\rho(\si)}\!\setminus{\sf k}_\lambda)\bu\si}\,.
$$
\end{lem}

\begin{proof}
Obviously, $\mathfrak L_\si\subset\bigcap_{\lambda\in\Lambda}\overline{(\Xi_{\rho(\si)}\!\setminus{\sf k}_\lambda)\bu\si}$\,. For the opposite inclusion, it is enough to find for every compact subset ${\sf k}$ of $\Xi_{\rho}$ an index $\mu\in\Lambda$ such that ${\sf k}\subset{\sf k}_\mu$\,. Indeed, ${\sf k}$ is covered by the family of interiors of the sets ${\sf k}_\lambda$ so, by compactness, it is also covered by a finite subfamily $\big\{{\sf k}^\circ_{\lambda_1},\dots{\sf k}^\circ_{\lambda_m}\big\}$\,. Since $\Lambda$ is directed, that index exists. So we have ${\sf k}\subset{\sf k}_\mu$ implying that 
$$
\overline{(\Xi_{\rho(\si)}\setminus{\sf k}_\mu)\bu\si}\subset\overline{(\Xi_{\rho(\si)}\setminus{\sf k})\bu\si}.
$$ 
The conclusion follows.
\end{proof}

\begin{ex}\label{bor1}
Consider the pair groupoid $\Xi=\Z^2$ over $\Z$\,. In this case $\Xi_n=\Z\times\{n\}$ can be exhausted by the family $\{\{-\lambda,\dots,\lambda\}\times \{n\}\}_{\lambda\in\mathbb N}$ so, by Lemma \ref{technicallemma}, 
$$
\mathfrak L_n=\bigcap_{\lambda\in\mathbb N}\{k\in\Z\mid |k|>\lambda\}=\emptyset\,.
$$
\end{ex}

\begin{prop}\label{eusper}
\begin{enumerate}
\item[(i)]
All the points in the orbit of $\si$ have the same limit set $\mathfrak L_\si$\,. One has 
\begin{equation}\label{inflorit}
\overline{\mathfrak O}_\si=\mathfrak O_\si\cup\mathfrak L_\si\,.
\end{equation} 
\item[(ii)]
If $\,\Xi$ is open and locally compact, the closed set $\mathfrak L_\si$ is invariant. 
\item[(iii)]
If the orbit of $\si$ is relatively compact, $\mathfrak L_\si$ is non-empty and it attracts the points of the orbit $\mathfrak O_\si$\,:  for every neighborhood $W$ of $\mathfrak L_\si$\,, there is a compact subset ${\sf k}$ of $\,\Xi_{\rho(\si)}$ such that 
\begin{equation}\label{zappa}
\big(\Xi_{\rho(\si)}\!\setminus\!{\sf k}\big)\bu\si\subset W.
\end{equation}
\end{enumerate}
\end{prop}

\begin{proof}
$(i)$ We observe that, for ${\sf k}\subset\Xi_{\rho(\xi\bu\si)}$\,, with rigorous computations based on the definitions
$$
\big(\Xi_{\rho(\xi\bu\si)}\!\setminus{\sf k}\big)\bu(\xi\bu\si)=\big[\big(\Xi_{\rho(\xi\bu\si)}\!\setminus{\sf k}\big)\xi\big]\bu\si=\big(\Xi_{\rho(\xi\bu\si)}\xi\setminus{\sf k}\xi\big)\bu\si=(\Xi_{\rho(\si)}\!\setminus{\sf k}\xi)\bu\si\,,
$$
implying that $\mathfrak L_{\,\xi\bu\si}\subset\mathfrak L_\si$ (because ${\sf k}\xi$ is compact), from which the first statement follows. 

\smallskip
The $\supset$ inclusion in \eqref{inflorit} is obvious. For each ${\sf K}\in\mathcal K(\Xi)$ one can write
$$
\overline{\mathfrak O}_\si=\overline{\Xi\bu\si}=\overline{\big({\sf K}\cup{\sf K}^{\rm c}\big)\bu\si}=({\sf K}\bu\si)\cup\overline{{\sf K}^{\rm c}\bu\si}\subset\mathfrak O_\si\cup\overline{{\sf K}^{\rm c}\bu\si}\,,
$$
from which $\subset$ follows. 

\smallskip
$(ii)$ Let $(\xi,\tau)\in\Xi\!\Join\!\Si$ with $\tau\in\mathfrak L_\si$\,. 
Using the convergence criterion $(iii)$ of Lemma \ref{labil}, there exists some divergent net $(\eta_i)_{i\in I}$ 
in $\Xi_{\rho(\si)}$ such that $\eta_i\bu\si\to\tau$. A priori, there is no reason to think that $(\xi,\eta_i)\in\Xi^{(2)}$.
But because of the openness of $\Xi$ and Fell's criterion \cite[Prop.\,1.1]{Wi} applied to the source map, there exists a net $(\xi_j)_{j\in J}$ and a subnet $(\eta_{i_j})_{j\in J}$ such that  $\d(\xi_j)=\r(\eta_{i_j})$ and $\xi_j\to\xi$\,, so one may write
$$
\xi\bu\tau=\lim_{j}\xi_j\bu(\eta_{i_j}\bu\si)=\lim_{j}\,(\xi_j\eta_{i_j})\bu\si\,.
$$
Since $\Xi$ is locally compact, $(\xi_j \eta_{i_j})_{j\in J}$ is also divergent.

\smallskip
$(iii)$ It is enough to show that, for a fixed open neighborhood $W$ of $\mathfrak L_\si$\,, there is a compact subset ${\sf k}$ of $\,\Xi_{\rho(\si)}$ such that \eqref{zappa} holds: if $\mathfrak L_\si$ were void, the empty set would be a neighborhood, which contradicts the inclusion ($\Xi$ is strongly non-compact). For any ${\sf k}\in\mathcal K(\Xi_{\rho(\si)})$\,, we set $\Si_\si({\sf k}):=\overline{\big(\Xi_{\rho(\si)}\!\setminus{\sf k}\big)\bu\si}$\,; complements will refer to $\overline{\mathfrak O}_\si$\,. Since $\mathfrak L_\si\subset W\cap\overline{\mathfrak O}_\si$\,, the family $\big\{\Si_\si({\sf k})^c\,\big\vert\,{\sf k}\in\mathcal K(\Xi_{\rho(\si)})\big\}$ is an open cover of the complement of $W\cap\overline{\mathfrak O}_\si$ in the compact space $\overline{\mathfrak O}_\si$\,. We extract a finite subcover  $\big\{\Si_\si({\sf k}_i)^c\,\big\vert\,i=1,\dots,n\big\}$\,. Then
$$
\Big(\Xi_{\rho(\si)}\!\setminus\bigcup_{i=1}^n{\sf k}_i\Big)\bu\si\subset\bigcap_{i=1}^n \overline{\big(\Xi_{\rho(\si)}\!\setminus{\sf k}_i\big)\bu\si}\subset W\cap\overline{\mathfrak O}_\si\subset W
$$
and the proof is finished, since a finite union of compact sets is compact.
\end{proof}

\begin{ex}\label{bor3}
We indicate now an example that is relevant for the problem of invariance of the limit sets. Let $X$ be a topological space, $\G$ a non-compact group and $\Pi\subset X^2$ an equivalence relation. The obvious product groupoid $\Xi=\Pi\times\G$ has unit space $\Si\equiv\Xi^{(0)}\!={\rm Diag}(X)\times\{\e\}\cong X$. The $\d$-fiber of $x$ is $\Xi_x\!=\{(y,x)\in\Pi\}\times\G$\,.

\smallskip
We prove now that $\mathfrak L_x=\overline{\mathfrak O}_x$\,. Let $y\in \overline{\mathfrak O}_x$ and let $(x_i)_{i\in I}\subset \mathfrak O_x$ be a net converging to $y$\,. As $\G$ is not compact, for every ${\sf k}\in\mathcal K(\G)$ there exists some $g_{\sf k}\in\G\setminus {\sf k}$\,. The net $(x_i,x,g_{\sf k})_{(i,{\sf k})\in I\times\mathcal K(\G) }$ belongs to $\Xi_x$\,, is divergent and fulfills 
$$
\lim_{i,{\sf k}}\,(x_i,x,g_{\sf k})\circ x=\lim_{i}x_i=y\,.
$$ 
Therefore $\mathfrak L_x=\overline{\mathfrak O}_x$\,, because of Lemma \ref{labil} and equation \eqref{inflorit}.

\smallskip
In particular, one can choose $X=\R$\,, $\G=\Z$\,. 
If $\Pi$ is given by 
\begin{equation*}\label{comeback}
x\,\Pi\,y\quad \Leftrightarrow\quad x,y\leq0 \textup{ \,or\, }x,y>0\,,
\end{equation*} 
one has $\mathfrak L_{1}=\overline{\mathfrak O}_{1}=[0,\infty)$\,, which is not invariant! Note that the equivalence relation is not open.
\end{ex}

\subsection{Recurrent points and wandering}\label{craciclor}

We keep the framework of the preceding subsection.

\begin{defn}\label{dransver}
When $\si\in\mathfrak L_\si$ holds, we say that $\si$ is \textit{a recurrent point}. We denote by $\Si_{\rm rec}^\th\equiv\Si_{\rm rec}$ the  family of all the recurrent points of the groupoid action.
\end{defn}

In \eqref{inflorit} the union could be disjoint or not.

\begin{prop}\label{recur}
For a point $\si\in\Si$\,, consider the following five conditions: 
\begin{enumerate}
\item[(a)] $\mathfrak O_\si\cap\mathfrak L_\si\ne\emptyset$\,,\quad (b) $\mathfrak L_\si=\overline{\mathfrak O}_\si$\,, \quad (c) $\si\in\mathfrak L_\si$\,, 
\item[(d)] there is a divergent net $(\xi_i)_{i\in I}\subset\Xi_{\rho(\si)}$ such that $\xi_i\bu\si\to\si$\,, 
\item[(e)] $\,\widetilde\Xi_\si^U$ is not relatively compact for any open neighborhood $U$ of $\si$\,.
\end{enumerate}
Then $(b)\Rightarrow(c)\Leftrightarrow(d)\Leftrightarrow(e)\Rightarrow(a)$\,. If $\,\Xi$ is open and locally compact, the five conditions are equivalent.
\end{prop}

\begin{proof}
The equivalence $(c)\Leftrightarrow(d)\Leftrightarrow(e)$ is the content of Lemma \ref{labil} for $\tau=\si$.

\smallskip
$(b)\Rightarrow(c)$ and $(c)\Rightarrow(a)$ are obvious.

\smallskip
To finish the proof, assume now that $\Xi$ is open and locally compact and we will show $(a)\Rightarrow(b)$. We know from Proposition \ref{eusper} (ii) that $\mathfrak L_\si$ is (closed and) invariant, so if (a) holds it contains the closure of the orbit $\mathfrak O_\si$. But it cannot be strictly bigger, by \eqref{inflorit}.
\end{proof}

\begin{cor}\label{laloc}
If $\,\Xi$ is open and locally compact, $\Si_{\rm rec}$ is invariant. 
\end{cor}

\begin{proof}
Since $\Xi$ is open and locally compact, we can describe its fellowship to $\mathfrak L_\si$ by condition $(b)$\,. Using Proposition \ref{eusper},\,(ii), for a recurrent point $\si$ and for $\xi\in\Xi_{\rho(\si)}$ we can write
$$
\mathfrak L_{\xi\bu\si}=\mathfrak L_\si=\overline{\mathfrak O}_\si=\overline{\mathfrak O}_{\xi\bu\si}\,,
$$
so $\xi\bu\si$ is also recurrent.
\end{proof}

\begin{defn}\label{uandering}
The point $\si\in\Si$ is {\it wandering} with respect to the action $(\Xi,\rho,\th,\Si)$ with strongly non-compact groupoid if $\si$ has a neighborhood $W$ such that $\widetilde\Xi^W_W$ is relatively compact. In the opposite case, we say that $\si$ is {\it non-wandering}. We denote by $\Si^\th_{\rm nw}\equiv\Si_{\rm nw}$ the family of all the non-wandering points. If $\Si^\th_{\rm nw}=\Si$ one says that {\it the action is non-wandering}.
\end{defn}

\begin{ex}\label{lalok}
If $(\G,\gamma,\Si)$ is a topological group dynamical system, one says that $\si\in\Si$ is wandering if ${\rm Rec}_\gamma(W,W)$ is relatively compact for some neighborhood $W$ of $\si$ (actually, this is mainly used for $\G=\Z,\R$)\,.  The discussion in Example \ref{bunul} shows that by Definition \ref{uandering} one gets the same concept.
\end{ex}

\begin{rem}\label{zabook}
Suppose that both $\Xi$ and $\Si$ are (Hausdorff and) locally compact. In this case wandering points are involved, with a different language, in \cite[Sect.\,2.1]{Wi}, extending the group case of \cite{Pal}. One says that the action is {\it proper} if the map 
$$
\Xi\!\Join\!\Si\ni(\xi,\si)\overset{\vartheta}{\to}(\xi\bu\si,\si)\in\Si\!\times\!\Si
$$ 
is proper, i.e.\! that inverse images of compact sets are compact. A strictly weaker concept is that of {\it a Cartan action}, for which each point $\si\in\Si$ has a compact neighborhood $W$ such that $\vartheta^{-1}(W\!\times\!W)$ is compact. In \cite[Lemma\,2.23]{Wi} it is shown that this happens precisely when all the points of $\Si$ are wandering (without mentioning the term). Consequently, in the locally compact case, for a Cartan action one has $\Si_{\rm nw}=\emptyset$\,. In addition \cite[Ex,\,2.1.19]{Wi} states that {\it for a Cartan action all the orbits are closed}.
\end{rem}

\begin{prop}\label{haihui}
The set $\Si_{\rm nw}$ is closed. If the groupoid $\Xi$ is open and locally compact, $\Si_{\rm nw}$ is also invariant. 
\end{prop}

\begin{proof}
Let $\tau\in\overline{\Si}_{\rm nw}$. By definition of closure, every open neighborhood $U$ of $\tau$ intersects $\Si_{\rm nw}$\,, so let $\si\in\Si_{\rm nw}\cap U$. As $U$ is a neighborhood of $\si\in\Si_{\rm nw}$\,, the set $\widetilde\Xi_U^U$ is not relatively compact. Thus $\tau\in\Si_{\rm nw}$\,.

\smallskip
Suppose that $\Xi$ is open and locally compact and let $\si\notin\Si_{\rm nw}$\,; we are going to show that $\eta\bu\si\notin\Si_{\rm nw}$\,, for all $\eta\in\Xi_{\rho(\si)}$\,. Let ${\sf N}_\eta$ a relatively compact open neighborhood of $\eta$ and let $W$ be an open neighborhood of $\si$ such that $\widetilde\Xi^W_W$ is relatively compact. Since $\Xi$ is open, ${\sf N}_\eta\bu W$ is an open neighborhood of $\eta\bu\si$. A plain application of the definitions leads to $\widetilde\Xi^{{\sf N}_\eta\bu W}_{{\sf N}_\eta\bu W}\subset{\sf N}_\eta\,\widetilde\Xi^W_W{\sf N}_\eta^{-1}$, which is relatively compact.
\end{proof}

\begin{ex}
The invariance may fail without openness.  Let $\Xi$ be the groupoid constructed in final part of Example \ref{bor3}, the one associated to the partition $\R=(-\infty,0]\cup(0,\infty)$ and to the group $\G=\Z$\,. Then consider the wide (locally compact) subgroupoid 
$$
\Delta:=\{(x,y,n)\in\Xi\mid x,y> 0\textup{ or }n=0\}.
$$ 
If we denote by $\th_1$ and $\th_2$ the canonical actions of $\Xi$ and $\Delta$ in $\Si=\R$\,, respectively, we have 
\begin{equation}\label{zuvertaj}
\Si_{\rm nw}^{\th_1}=\R \quad\textup{and}\quad\Si_{\rm nw}^{\th_2}=[0,\infty)\,.
\end{equation}
To check the second equality, note that $\widetilde\Delta_W^W=W\!\times\! W\!\times\!\{0\}$ if $W\subset(-\infty,0)$\,, but $\widetilde\Delta_W^W=W\!\times\! W\!\times\!\,\Z$ when $W\subset(0,\infty)$\,. The second set in \eqref{zuvertaj} is not invariant, since $0$ is orbit-equivalent with any negative number. In addition, it follows easily that 
$$
\mathfrak L_x^{\theta_2}=
\left\{\begin{array}{ll}
[0,\infty)     & \textup{if\ } x>0 \\
\emptyset      & \textup{if\ } x\leq0\,. \\
\end{array}\right.
$$ 
So  
$\Si_{\rm rec}^{\th_2}=(0,\infty)$ which {\it is not closed}. 
\end{ex}

\begin{prop}\label{lalik}
One has
\begin{equation}\label{ocompleta}
\Si_{\rm rec}\subset\bigcup_{\si\in\Si}\mathfrak L_\si\subset\Si_{\rm nw}\,.
\end{equation}
\end{prop}

\begin{proof}
The first inclusion follows from the definition of recurrent points. So we only need to prove that $\mathfrak L_\si\subset\Si_{\rm nw}$\,. Pick $\tau\in\mathfrak L_\si$ and let $U$ be a neighborhood of $\tau$. By Lemma \ref{labil}, we already know that $\widetilde\Xi_\si^U$ is not relatively compact. If $\xi\in\widetilde\Xi_\si^U$, then $\xi\bu\si\in U$; using Lemma \ref{ushor} we get 
$$
\widetilde\Xi_\si^U\xi^{-1}\!=\widetilde\Xi_{\xi\bu\si}^U\subset\widetilde\Xi_U^U\,,
$$
showing that the latter set is not relatively compact. (Notice that $\d\big(\widetilde\Xi_\si^U \big)=\{\r(\xi^{-1})\}$, hence $\widetilde\Xi_\si^U\xi^{-1}$ is relatively compact if and only if $\widetilde\Xi_\si^U$ is relatively compact)
\end{proof}

\begin{cor}\label{prinurmare}
If at least one of the orbits is relatively compact (in particular if $\,\Si$ is compact), $\Si_{\rm nw}$ is non-void.
\end{cor}

\begin{proof}
This follows from \eqref{ocompleta} and Proposition \ref{eusper}\,(iii).
\end{proof}

\begin{prop}\label{hayhuy}
If $\,\Si$ is compact, $\Si_{\rm nw}$ attracts the points of $\,\Si$\,: for every $\si\in\Si$ and for every neighborhood $V$ of $\,\Si_{\rm nw}$\,, one has $\xi\bu\si\in V$ for every $\xi\in\Xi_{\rho(\si)}$ outside some compact set.
\end{prop}

\begin{proof}
One has to show that for every neighborhood $V$ of $\,\Si_{\rm nw}$\,, the complement in $\Xi_{\rho(\si)}$ of the set $\,\widetilde\Xi^V_\si$ is relatively compact. If $V$ is a neighborhood of $\Si_{\rm nw}$\,, by \eqref{ocompleta}, it is also a neighborhood of the limit set $\mathfrak L_\si$\,. One applies Proposition \ref{eusper}\,(iii) to infer that there exists a compact subset ${\sf k}$ of $\Xi_{\rho(\si)}$ such that $\big(\Xi_{\rho(\si)}\!\setminus\!{\sf k}\big)\bu\si\subset V$, i.\,\,e. $\Xi_{\rho(\si)}\!\setminus\!{\sf k}\subset \widetilde\Xi^V_\si$. Then the complement of $\widetilde\Xi^V_\si$ in $\Xi_{\rho(\si)}$ is contained in ${\sf k}$ and the proof is finished.
\end{proof}

\begin{ex}\label{aminte}
If $\Xi\!\ltimes_\th\!\Si$ is the action groupoid of the groupoid action $(\Xi,\th,\rho,\Si)$\,, for every $\si\in\Si$ the two limit sets that make sense are equal. The recurrent sets are also equal. These are easy consequences of the definitions. Formula \eqref{consequently} shows that if $\si$ is non-wandering for the action $\th$, then it is also non-wandering for the action groupoid $\Xi\!\ltimes_\th\!\Si$\,. If $\Si$ is compact, the non-wandering sets coincide.
\end{ex}

\begin{ex}\label{aducere}
We refer now to the pull-back construction of \ref{cabalier}, assuming that $\O$ is locally compact. We leave it to the reader to check the formula
\begin{equation*}\label{rellorb}
\mathfrak L^{\downarrow\downarrow}_{\o}=h^{-1}\big(\mathfrak L_{h(\o)}\big)\,,\quad\forall\,\o\in\O\,,
\end{equation*}
where the limit set in the l.\,h.\,s. is computed in $h^{\downarrow\downarrow}(\Xi)$\,, being a subset of $\O$\,. It follows easily that $\o\in\O$ is $h^{\downarrow\downarrow}(\Xi)$-recurrent if and only if $h(\o)\in X$ is $\Xi$-recurrent. Using \eqref{transfers}, one checks easily that $\o$ is $h^{\downarrow\downarrow}(\Xi)$-wandering if and only if $h(\o)$ is $\Xi$-wandering. 
\end{ex}

\section{Minimality and almost periodicity}\label{calorifier}

\subsection{Fixed points}\label{golorofer}

Let $(\Xi,\rho,\th,\Si)$ be a continuous groupoid action.

\begin{defn}\label{intaresc}
A {\it fixed point} is a point $\si\in\Si$ such that $\xi\bu\si=\si$ for every $\xi\in\Xi_{\rho(\si)}$\,. This is equivalent to $\widetilde\Xi_\si^\si=\Xi_{\rho(\si)}$\,. We write $\si\in\Si^\th_{\rm fix}\equiv\Si_{\rm fix}$\,.
\end{defn}

\begin{ex}\label{pustan}
For group actions, converted into groupoid actions as in Exemple \ref{bunul}, the notion of fixed point boils down to the usual one. The same can be said if the model is that of Example \ref{starlet}.
\end{ex}

\begin{ex}\label{puskan}
If $\,\underline\Si\subset\Si$ is an open set, then $\underline\Si^{\underline{\th}}_{\rm fix}\supset\underline{\Si}\cap\Si^\th_{\rm fix}$\,. We used the restriction notions from Subsection \ref{creasta}.
\end{ex}

\begin{ex}\label{beltita}
The point $y\in Y$ is said to be {\it a fixed point of the partial group action} of Definition \ref{iexel} if $\beta_a(y)=y$ for $a\in\G$ and $y\in Y_{a^{-1}}$\,. This happens if and only if $y$ is a fixed point with respect to the canonical action of the groupoid $\G\!\ltimes_{[\beta]}\!Y$.
\end{ex}

\begin{ex}\label{ingrid}
Let $\Xi(\nu)$ be the Deaconu-Renault groupoid attached to the local homeomorphism $\nu:X\to X$. The unit $y$ is a fixed point if and only if for every $x\ne y$ there are no positive integers $k,l$ with $\nu^k(x)=\nu^l(y)$\,. 
\end{ex}

\begin{ex}\label{partiszan}
An element $\o\in\O$ is a fixed point of the pull-back groupoid $h^{\downarrow\downarrow}(\Xi)$ introduced in subsection \ref{cabalier} if and only if $h(\o)$ is a fixed point of the canonical action of $\Xi$ on its unit space.
\end{ex}

\begin{prop}\label{zets}
The set $\Si_{\rm fix}$ is invariant. If $\,\Xi$ is open, $\Si_{\rm fix}$ is also closed.
\end{prop}

\begin{proof}
If $\si\in\Si_{\rm fix}$, then $\mathfrak O_\si=\{\si\}$, which makes $\Si_{\rm fix}$ trivially invariant. 

\smallskip
If $\Xi$ is open and $\si\in\overline{\Si}_{\rm fix}$, then exists a net $(\si_i)_{i\in I}$ of fixed points converging to $\si$. Let $\xi\in\Xi_{\rho(\si)}$. By Fell's criterion \cite[Prop.\,1.1]{Wi}, applied to the open map $\d$ and the net $\rho(\si_{i})\in X$, there exists a net $(\xi_j)$ and a subnet $(\si_{i_j})$ such that $\xi_j\to\xi$ and $\rho(\si_{i_j})=\d(\xi_j)$\,. By continuity, we have 
$$
\xi\bu\si=(\lim_j \xi_j)\bu(\lim_j \si_{i_j})=\lim_j (\xi_j\bu\si_{i_j})=\lim_j \si_{i_j}=\si.
$$ 
So $\si$ is a fixed point for the action.
\end{proof}

Examples \ref{maidevreme} and \ref{valioso4} will illustrate that openness of $\Xi$ is important.

\subsection{Periodic and almost periodic points}\label{kalorifier}

Although it is not always necessary, in this subsection we prefer to assume that in the continuous action $(\Xi,\rho,\th,\Si)$ the groupoid $\Xi$ is strongly non-compact (all the $\d$-fibers are non-compact).

\begin{defn}\label{indrasniesc}
\begin{enumerate}
\item[(a)] Let $x\in X$. The subset ${\sf A}$ of $\,\Xi_x$ is called {\it syndetic} if ${\sf KA}=\Xi_x$ for a compact subset ${\sf K}$ of $\,\Xi$\,.
\item[(b)] We say that $\si\in\Si$ is {\it periodic}, and we write $\si\in\Si_{\rm per}$\,, if $\,\widetilde\Xi_\si^\si$ is syndetic in $\Xi_{\rho(\si)}$\,. 
\item[(c)] The point $\si$ is called \textit{weakly periodic} (we write $\si\in\Si_{\rm wper}$) if the subgroup $\widetilde\Xi_\si^\si$ is not compact. 
\item[(d)] The point $\si\in\Si$ is said to be {\it almost periodic} if $\,\widetilde\Xi_\si^U$ is syndetic in $\Xi_{\rho(\si)}$ for every neighborhood $U$ of $\si$ in $\Si$\,.
We denote by $\Si_{\rm alper}$ the set of all the almost periodic points. If $\Si_{\rm alper}=\Si$\,, the action is {\it pointwise almost periodic}.
\end{enumerate}
\end{defn}

\begin{ex}\label{maidevreme}
Consider the equivalence relation $\Pi$ introduced in Remark \ref{tutosh}, and form the groupoid product $\Xi=\Pi\times \Z$\,. By considering the action of $\Xi$ in $\Si=\R$\,, we see that 
$$
\Si_{\rm fix}=(-\infty, -1)\cup(1,\infty)\,,\quad \Si_{\rm per}=\Si_{\rm wper}=\Si_{\rm alper}=\R\,,
$$ 
which also illustrates that, in general, the set $\Si_{\rm fix}$ is not closed. 
\end{ex}

\begin{ex}\label{sinfin}
We fix a unit $y\in X$ of the Deaconu-Renault groupoid of the local homeomorphism $\nu:X\to X$. We remarked that \eqref{surjdr} becomes a bijection for singletons. Then $y$ is periodic if and only if it is weakly periodic, and this happens exactly when $\Z_\nu(x,x)$ (a subgroup of $\Z$) does not coincide with $\{0\}$ (then it will be both infinite and syndetic). This means that for some positive integers $k\ne l$ one has $\nu^k(x)=\nu^l(x)$\,.
\end{ex}

\begin{ex}\label{constantinescu}
Let $\Xi[\beta]$ be the groupoid associated to the partial action $\beta$ of the discrete group $\G$ on the topological space $Y$. After \eqref{debilla}, we established an identification between $\Xi[\beta]_y^N$ and ${\rm Rec}_{[\beta]}(y,N)$ for every point $y\in Y$ and subset $N\subset Y$. This allows rephrasing the periodicity properties in this case. For example, $y$ is periodic in the groupoid $\Xi[\beta]$ if and only if the subgroup of $\G$ composed of all the elements "defined at $y$ and leaving it invariant" is syndetic (which is equivalent to having finite index in $\G$).
\end{ex}

\begin{prop}\label{minwan}
Let $\Xi$ be a strongly non-compact groupoid. Then one has
\begin{equation}\label{pastish}
\Si_{\rm fix}\overset{(1)}{\subset}\Si_{\rm per}\overset{(2)}{\subset}\Si_{\rm wper}\cap\Si_{\rm alper}\overset{(3)}{\subset}\Si_{\rm wper}\cup\Si_{\rm alper}\overset{(4)}{\subset}\Si_{\rm rec}\overset{(5)}{\subset}\bigcup_{\si\in\Si}\mathfrak L_\si\overset{(6)}{\subset}\Si_{\rm nw}\,.
\end{equation} 
\end{prop}

\begin{proof} 
The inclusions (1) and (3) are obvious. We proved (5) and (6) previously, in Proposition \ref{lalik}.

\smallskip
To deduce (2) from the definitions, note that $\widetilde\Xi_\si^\si\subset\widetilde\Xi_\si^U$ if $\si\subset U$ and that a syndetic set is not compact (since the $\d$-fibers are not compact).

\smallskip
The inclusion (4) also follows easily from the definitions, by the same type of arguments: use Proposition \ref{recur}\,(e) to describe recurrent points.
\end{proof}

\begin{ex}\label{partsim}
The situation is particularly simple for equivalence relations, outlined in subsection \ref{castandyu}, especially because of equation \eqref{uncazz}. In particular, for $y\in X\equiv\Si$ one gets 
$$
\widetilde\Xi_y\cong\{x\in X\!\mid\!x\,\Pi\,y\}\quad{\rm and}\quad \widetilde\Xi_y^y=\{(y,y)\}\cong\{y\}\,.
$$
To insure that the associated groupoid is strongly non-compact, we require that for every $y\in X$ the set $\{x\in X\!\mid\!x\,\Pi\,y\}$ is non-compact. Then $\Si_{\rm wper}=\emptyset$ (so there are no fixed points or periodic points). With some abuse of notation and interpretation, $y$ will be almost periodic if and only if $\{x\in U\!\mid\!x\,\Pi\,y\}$ is syndetic in $\{x\in X\!\mid\!x\,\Pi\,y\}$ for every neighborhood $U$ of $y$\,. If $X$ is locally compact, one could choose a relatively compact neighborhood and syndeticity contradicts the fact that the fiber in $y$ is non-compact. Therefore, in the locally compact strongly non-compact case, equivalent relations do not exhibit almost periodic points.

\smallskip
If the equivalence relation is not forced to lead to a strongly non-compact groupoid, the situation might be very different. In particular, some of the inclusions in \eqref{pastish} no longer hold. We leave this to the reader.
\end{ex}

\begin{ex}\label{portiszan}
Denoting by $\beta$ any of the properties "periodic", "weakly periodic" and "almost periodic", it is easy to check that $\o\in\O$ has $\beta$ in the pull-back $h^{\downarrow\downarrow}(\Xi)$ if and only if $h(\o)\in X$ has $\beta$ in $\Xi$\,. This happens mostly because of equality \ref{transfers}.
\end{ex}

\begin{prop}\label{sets}
The sets $\Si_{\rm per},\Si_{\rm wper}$ are invariant. 
\end{prop}

\begin{proof}
By Lemma \ref{ushor}, one has $\widetilde\Xi_{\xi\bu\si}^{\xi\bu\si}=\xi\,\widetilde\Xi_\si^\si\xi^{-1}$, so $\Si_{\rm wper}$ is invariant. 

\smallskip
We focus now on $\Si_{\rm per}$\,; let $\si\in\Si_{\rm per}$ and $\xi\in\Xi_{\rho(\si)}$\,. For some compact set ${\sf K}$\,, we have 
$$
{\sf K}\,\xi^{-1}\widetilde\Xi_{\xi\bu\si}^{\xi\bu\si}\,\xi={\sf K}\,\widetilde\Xi_\si^\si=\Xi_{\rho(\si)}=\Xi_{\rho(\xi\bu\si)}\xi\,.
$$ 
Hence
$$
({\sf K}\xi^{-1})\widetilde\Xi_{\xi\bu\si}^{\xi\bu\si}=\Xi_{\rho(\xi\bu\si)}\,,
$$ 
meaning that $\widetilde\Xi_{\xi\bu\si}^{\xi\bu\si}$ is syndetic in $\Xi_{\rho(\xi\bu\si)}$\,, and thus $\xi\bu\si\in \Si_{\rm per}$ holds.
\end{proof}

The sets of periodic, weakly periodic or almost periodic points might fail to be closed, even for group actions.

\subsection{Compact orbits}\label{coproduct}

We connect now periodicity with the notion of a periodicoid point, introduced in \cite[Section 3.1]{BBdN} in a more restricted context.

\begin{prop}\label{lacidraci}
Every periodic point $\si$ has a compact orbit. 
\end{prop}

\begin{proof}
If $\si\in\Si_{\rm per}$\,, then ${\sf K}\widetilde\Xi_\si^\si=\Xi_{\rho(\si)}$ for some compact set ${\sf K}\subset \Xi$\,. Any net $(\si_i)\subset \mathfrak O_\si$ can be written as $\si_i=(\xi'_i\xi_i)\bu\si$ for some nets $(\xi'_i)\subset{\sf K} $\,, $(\xi_i)\subset\widetilde\Xi_\si^\si$\,. By compactness, extract a subnet $(\xi'_j)\subset (\xi'_i) $ such that $\xi'_j\to \xi'\!\in {\sf K}$ and get 
$$
\si_j\!:=(\xi'_j\xi_j)\bu\si=\xi'_j\bu(\xi_j\bu\si)=\xi'_j\bu\si\to\xi\bu\si\in\mathfrak O_\si\,.
$$ 
We conclude that $\mathfrak O_\si$ is compact.
\end{proof}

\begin{proof}
Let us indicate a second proof, using a construction that will also be useful for Proposition \ref{dracilaci}. For $\si\in\Si$\,, let us define the continuous surjective function
\begin{equation*}\label{neosofist}
\alpha^\si\!:\Xi_{\rho(\si)}\!\to\mathfrak O_\si\subset\Si\,,\quad\alpha^\si(\xi):=\th_\xi(\si)\equiv\xi\bu\si.
\end{equation*}
If $\si\in\Si_{\rm per}$\,, then ${\sf K}\widetilde\Xi_\si^\si=\Xi_{\rho(\si)}$ for some compact set ${\sf K}\subset \Xi$\,. By using the definition of $\widetilde\Xi_\si^\si$\,, one gets
$$
\mathfrak O_\si=\alpha^\si\big(\Xi_{\rho(\si)}\big)=\alpha^\si({\sf K})\,,
$$
which is compact, as a direct continuous image of a compact set.
\end{proof}

For deriving a converse proposition, we will use a well known lemma (with proof, for the convenience of the reader):

\begin{lem}\label{groups}
Let $\,\Xi=\G$ be a locally compact, second countable group acting on a topological space $\Si$. If $\si\in\Si$ has a compact orbit, then $\widetilde{\G}_\si^\si\equiv{\rm Rec}(\si,\si)$ (see subsection \ref{castyanu}) is syndetic in $\G$\,.
\end{lem}

\begin{proof}
Let ${\sf N}_g\subset \G$ be a relatively compact, open neighborhood of $g$\,, for every $g\in \G$\,. Observe that 
$$
\G=\G\,\widetilde\G_\si^\si= \Big(\bigcup_{g\in\G} {\sf N}_g\Big)\,\widetilde\G_\si^\si.
$$ 
So 
$$
\mathfrak O_\si=\G\bu\si=\Big(\bigcup_{g\in\G} {\sf N}_g\Big)\widetilde\G_\si^\si\bu\si= \bigcup_{g\in\G}{\sf N}_g\bu\si.
$$ 
Let us, for the sake of the argument, use (and prove it later) that the set ${\sf N}_g\bu\si$ is a neighborhood of $\si$. By compactness we can extract a finite index set ${\sf F}=\{g_1,\ldots,g_n\}$ such that 
$$
\mathfrak O_\si=\bigcup_{i=1}^n{\sf N}_{g_i}\!\bu\si\subset \Big(\bigcup_{i=1}^n\overline{\sf N}_{g_i}\Big)\bu \si.
$$ 
Define ${\sf K}=\bigcup_{i=1}^n\overline{\sf N}_{g_i}$ and notice that for every $g\bu\si\in\mathfrak O_\si$, there  exists $h\in {\sf K}$ such that 
$$
g\bu\si=h\bu\si\Rightarrow (h^{-1}g)\bu\si=\si,
$$ 
meaning that $h^{-1}g\in\widetilde{\G}_\si^\si$ and $g\in {\sf K}\widetilde{\G}_\si^\si$. As ${\sf K}$ is compact, we conclude that $\widetilde{\G}_\si^\si$ is syndetic in $\G$. 

\smallskip
Now, we fill the remaining gap: Pick $W\subset \G$, another relatively compact, open neighborhood of $g$ but satisfying $\overline{W}\subset \sf N_g$. As $\G$ is second countable, it has the Lindelöf property, so exists a countable subset $\sf C\subset \G$ making $\G=\sf CW$ true. If ${\sf N}_g\bu\si$ had void interior, so would $c\bu\big[W\bu\si\big]=\big[cW\big]\bu\si$. But then we could decompose $\mathfrak O_\si$ as a countable union of closed nowhere dense sets: $$\mathfrak O_\si=\G\bu\si=\sf CW\bu\si=\bigcup_{c\in{\sf C}} \big[c\overline W\big]\bu\si, $$ contradicting Baire's category theorem.
\end{proof} 

\begin{prop}\label{dracilaci}
Assume that $\Xi$ is locally compact, second countable and open. If $\si\in\Si$ has a compact orbit, then it is a periodic point.
\end{prop}

\begin{proof}
Let us first treat the case of the canonical action of $\Xi$ in its unit space $X$ (and notice that $X$, being closed in $\Xi$\,, is a locally compact, second countable space by its own). We will use a small amount of information from \cite[Sect.\,22]{Wi}, treating the Mackey-Glimm-Ramsay Dichotomy for groupoids; see also \cite{Ra}. For $x\in X$, let us define the continuous surjective function
\begin{equation*}\label{nesofist}
\alpha^x\!:\Xi_{x}\!\to\mathcal O_x\subset X\,,\quad\alpha^x(\xi):=\xi\circ x\,.
\end{equation*}
One has $\alpha^x(\xi)=\alpha^x(\eta)$ if and only if $\xi^{-1}\eta\in\Xi_x^x$\,. This leads to a continuous bijection 
\begin{equation}\label{homeomorphism}
\widetilde\alpha^x\!:\Xi_{x}/\,\Xi^x_x\to\mathcal O_x\,.
\end{equation} The quotient map $p:\Xi_{x}\to\Xi_{x}/\,\Xi^x_x$ is (surjective, continuous and) open, cf.\;\cite[Ex.\,2.2.1]{Wi}. Then $\widetilde\alpha^x$ is a homeomorphism if and only if $\alpha^x$ is an open function. By \cite{Wi} (see the non-trivial implication $(2)\Rightarrow({\rm e})$ on pages 41-42), this happens if the orbit $\mathcal O_x$ is Baire. In our case this is insured, since it is (Hausdorff and) compact. So we conclude that \eqref{homeomorphism} is a homeomorphism and thus $\Xi_{x}/\,\widetilde\Xi^x_x$ is compact.

\smallskip
To finish this part of the proof, we show now that the compactness of $\Xi_{x}/\,\widetilde\Xi^x_x$ implies that $\Xi^x_x$ is syndetic in $\Xi_{x}$\,. Let $\big\{{\sf V}_i\,\big\vert\,i\in I\big\}$ be a covering of the locally compact space $\Xi_{x}$ by open subsets with compact closures. Since $p$ is open and surjective, $\big\{W_i:=p({\sf V_i})\,\big\vert\,i\in I\}$ will be an open covering of the compact space $\Xi_{x}/\,\Xi^x_x$\,, from which we extract a finite subcovering $\big\{W_i:=p({\sf V_i})\,\big\vert\,i\in I_0\big\}$\,. Then ${\sf V}\!:=\bigcup_{i\in I_0}\!{\sf V_i}$ has a compact closure ${\sf K}$ such that $p({\sf K})=\Xi_{x}/\,\Xi^x_x$\,. To check that ${\sf K}\,\Xi^x_x=\Xi_{x}$\,, pick $\xi\in\Xi_{x}$\,. Then $p(\xi)=p(\eta)$ for some $\eta\in{\sf K}$\,. By the definition of $p$\,, this means $\xi\in\eta\,\Xi^x_x$ and we are done. 

\smallskip
Now, we will derive the full result: Suppose that $\Xi$ acts on a very general space $\Si$ and $\mathfrak O_\si\subset\Si$ is compact, for some $\si\in\Si$\,. As the anchor map is continuous, $\rho(\mathfrak O_\si)=\mathcal O_{\rho(\si)}\!\subset X$ (Lemma \ref{lavila}) is compact and by the previous discussion, the decomposition $\Xi_{\rho(\si)}={\sf K_1}\,\Xi_{\rho(\si)}^{\rho(\si)}$ holds for some compact set ${\sf K_1}\subset\Xi$\,. Remark that $\G=\Xi_{\rho(\si)}^{\rho(\si)}$ is a (locally compact) group acting continuously on the compact space
$$
\mathfrak O_\si\cap\{\tau\in\Si\mid \rho(\tau)=\rho(\si)\}\,.
$$ 
 By applying Lemma \ref{groups} (with a change of notations), we obtain another compact set ${\sf K_2}\subset\Xi$ such that 
$$
\Xi_{\rho(\si)}={\sf K_1}\Xi_{\rho(\si)}^{\rho(\si)}={\sf K_1}\big({\sf K_2}\widetilde\Xi_\si^\si\big)=\big({\sf K_1}{\sf K_2}\big)\widetilde\Xi_\si^\si\,,
$$
 finishing the proof.
\end{proof}

There is a shorter proof, in only one step and avoiding the use of Lemma \ref{lacidraci}, but it requires $\Si$ to be locally compact and second countable, which we succeeded to avoid.

\subsection{Minimal sets}\label{cralorifier}

Minimality is a very important property in classical topological dynamics; it extends straightforwardly to groupoid actions, denoted below by $(\Xi,\rho,\th,\Si)$\,. During this subsection, both $\Xi$ and $\Si$ are assumed to be locally compact.

\begin{defn}\label{niminal}
A closed invariant subset $M\subset \Si$ is called {\it minimal} if it does not contain proper non-void closed invariant subsets. Equivalently, $M$ is {\it minimal} if all the orbits contained in $M$ are dense in $M$. {\it The action is minimal} if $\Si$ itself is minimal.
\end{defn}

The minimal sets are the closed invariant non-empty subsets of $\Si$ which are minimal under such requirements. Two minimal sets either coincide or are disjoint. Clearly {\it transitivity} $\Rightarrow$ {\it minimality} $\Rightarrow$ {\it pointwise transitivity) and the implications are strict in general (even for group actions)}.

\begin{rem}\label{diversus}
For open groupoids, minimal sets are either clopen or nowhere dense. This follows from Proposition \ref{intclo}: the boundary of an invariant set is invariant. So if $\Xi$ is open and $M$ is minimal, the boundary $\partial M\subset M$ is closed and invariant, therefore it should be void (i.e.\,$M$ is open) or coincide with $M$ (meaning that $M$ is nowhere dense). But in general this is not the case: in Remark \ref{tutosh}, $[-1,1]\subset\mathbb R$ is a minimal set. 
\end{rem}

\begin{rem}\label{diverticius}
A function $\varphi:\Si\to\R$ is called {\it invariant} with respect to the groupoid action if $\varphi(\si)=\varphi(\tau)$ whenever $\si\overset{\th}{\sim}\tau$\,. By an obvious proof, one shows that if the action is minimal and $\varphi$ is continuous at least at one point, it has to be constant.
\end{rem}

We proceed now to characterize minimality. 

\begin{prop}\label{ipaschuli}
Let $\emptyset\ne M\subset \Si$ be closed and invariant. Then $M$ is minimal if and only if for every $U\subset \Si$ open, with $U\cap M\ne\emptyset$\,, one has $M={\rm Sat}(U\cap M)$\,. 
\end{prop}

\begin{proof}
{\bf if:} Assume that $\emptyset\ne N\subset M$ is closed and invariant; then $U= N^c\!=\Si\setminus N$ is open. If we show that $N^c\cap M=\emptyset$ one gets $M=N$, i.e. the minimality of $M$. But $N^c\cap M=\emptyset$ follows if we check that $M\not\subset{\rm Sat}(N^c\cap M)$\,, by assumption. This would follow from $N\cap{\rm Sat}(N^c\cap M)=\emptyset$\,. But $\si\in N\cap{\rm Sat}(N^c\cap M)$ means that $\si\in N$ and it is in the orbit of some element of $N^c$. This is impossible since $N$ is invariant and $N^c$ is its complement. 

\smallskip
{\bf only if:} Assuming now that $M$ is minimal, for every $\si\in M$ one has $M=\mathfrak Q_\si=\overline{\mathfrak O}_\si$\,. If $U\cap M\ne\emptyset$\,, $U$ being open, we have $U\cap\mathfrak O_\si\ne\emptyset$\,. But this means that $\si$ is in the orbit of some point that belongs to $U\cap\mathfrak O_\si\subset U\cap M$, so $x\in{\rm Sat}(U\cap M)$\,.
\end{proof}

The main result of this section is 

\begin{thm}\label{sacacarez}
If the point $\si\in\Si$ is almost periodic, its orbit closure $\overline{\mathfrak O}_\si$ is minimal and compact. If, in addition, the groupoid $\Xi$ is open, $\si$ is almost periodic if and only if $\,\overline{\mathfrak O}_\si$ is minimal and compact.
\end{thm}

\begin{proof}
Suppose that $\si$ is almost periodic; we show first that its orbit closure $\overline{\mathfrak O}_\si$ is compact. Let $U_0$ be a compact neighborhood of $\si$. Using the assumptions, for some compact set ${\sf K}$ one has
$$
\mathfrak O_\si=\Xi_{\rho(\si)}\bullet\si=\big({\sf K}\,\widetilde \Xi_{\si}^{U_0}\big)\bullet\si={\sf K}\bullet\big(\widetilde \Xi_{\si}^{U_0}\bullet\si\big)\subset{\sf K}\bullet U_0 ={\rm compact}\,,
$$
so $\overline{\mathfrak O}_\si$ is compact.

\smallskip
If $\overline{\mathfrak O}_\si$ is not minimal, it strictly contains a minimal (and compact) set $M$. The point $\si$ does not belong to $M$, so there are disjoint open sets $U,V\subset\Si$ such that $\si\in U$ and $M\subset V$. For an arbitrary compact set ${\sf K}\subset\Xi$ we will now show that ${\sf K}\,\widetilde \Xi_{\si}^U\ne\Xi_{\rho(\si)}$\,, implying that in fact $\si$ is not almost periodic.

\smallskip
The set $M$ being invariant, ${\sf K}^{-1}\!\bullet M\subset M$ holds. Let $W$ be a neighborhood of $M$ with ${\sf K}^{-1}\!\bullet W\subset V$. Since $M\subset\overline{\mathfrak O}_\si=\overline{\Xi_{\rho(\si)}\!\bullet\si}$\,, there exists $\eta\in\Xi_{\rho(\si)}$ such that $\eta\bullet\si\in W$. Then 
$$
\big({\sf K}^{-1}\eta\big)\bullet\si={\sf K}^{-1}\bullet(\eta\bullet\si)\subset{\sf K}^{-1}\!\bullet W\subset V,
$$ 
and thus $\big({\sf K}^{-1}\eta\big)\bullet\si$ is disjoint from $U$, meaning that ${\sf K}^{-1}\eta$ is disjoint from $\widetilde \Xi_{\si}^U$. This shows that 
$$
\Xi_{\rho(\si)}\ni\eta\notin{\sf K}\,\widetilde \Xi_{\si}^U\!\ne\Xi_{\rho(\si)}\,,
$$ 
finishing the proof.

\smallskip
For the converse, suppose now that $\d$ is open and $\overline{\mathfrak O}_\si=\overline{\Xi_{\rho(\si)}\!\bullet\si}$ is minimal and compact. Let $U$ be an open neighborhood of $\si$. For each $\xi\in\Xi$\,, choose an open neighborhood ${\sf N}_\xi$ of $\xi$ with compact closure. The sets $\Xi\bullet U$ and ${\sf N}_\xi\bullet U$ are open in $\Si$\,, cf. \cite[Ex.\,2.1.11]{Wi}. By minimality one has 
$$
\overline{\mathfrak O}_\si\!\subset\Xi\bullet U=\bigcup_{\xi\in\Xi} \xi\bullet U\subset\bigcup_{\xi\in\Xi} {\sf N}_\xi\bullet U.
$$ 
By compactness of $\overline{\mathfrak O}_\si$ applied to the open cover above, for a finite set ${\sf F}=\{\xi_1,\dots,\xi_k\}\subset\Xi$ we get 
$$
\overline{\mathfrak O}_\si\subset\bigcup_{i=1}^k {\sf N}_{\xi_i}\!\bullet U.
$$ 
If $\eta\in\Xi_{\rho(\si)}$ then $\eta\bullet\si\in{\sf N}_{\xi_j}\!\bullet U$ for some $j$, and then one has $\eta\bullet\si\in\eta_j\bu U$ for some $\eta_j\in{\sf N}_{\xi_j}$. It follows immediately that $\r(\eta)=\r(\eta_j)=\d\big(\eta_j^{-1}\big)$ and
$$
\eta_j^{-1}\!\bullet(\eta\bullet\si)=\big(\eta_j^{-1}\eta\big)\bullet\si\in U.
$$ 
This means that $\eta_j^{-1}\eta\in\widetilde \Xi_{\si}^U$ or, equivalently, that 
$$
\eta\in\eta_j\,\widetilde \Xi_{\si}^U\subset{\sf N}_{\xi_j}\widetilde \Xi_{\si}^U\subset\bigcup_{i=1}^k\Big(\overline{{\sf N}_{\xi_i}}\,\widetilde \Xi_{\si}^U\Big)=\Big(\bigcup_{i=1}^k\overline{{\sf N}_{\xi_i}}\Big)\,\widetilde \Xi_{\si}^U.
$$ 
Since $\eta$ is arbitrary one gets $\Xi_{\rho(\si)}\!\subset{\sf K}\,\widetilde \Xi_{\si}^U$, where ${\sf K}\!:=\bigcup_{i=1}^k\!\overline{{\sf N}_{\xi_i}}$ is compact. We checked that $\widetilde \Xi_{\si}^U$ is syndetic in $\Xi_{\rho(\si)}$\,, so $\si$ is almost periodic. 
\end{proof}

\begin{ex}\label{dinasist}
In the case of the transformation groupoid associated to a topological dynamical system $(\G,\gamma,X)$\,, one recovers the classical result (\cite[pag.11]{Au} and \cite[pag.\,28,\,38,\,39]{EE}). For this, we use Example \ref{bunul}. First of all, it is clear that minimality of the group action coincides with minimality in the sense of groupoids, since the orbits are the same. The relevant recurrence sets also coincide: set $S=\{\si\}$ and $T=U$ in \eqref{bunika}. Finally, syndeticity has the same meaning in the two cases.
\end{ex}

\begin{ex}\label{cocoselu}
Consider the case of the pair groupoid $\Xi:=X\!\times\! X$ acting on its unit space $X$, taken to be compact. The action is transitive (only one orbit), so every $x\in X$ should be almost periodic. And it is, since $\Xi_x=X\!\times\!\{x\}$ and $\widetilde\Xi_x^U=U\!\times\!\{x\}$\,. The set ${\sf K}:=X\!\times\!\{x\}$ itself is compact, and 
$$
{\sf K}\,\widetilde\Xi_x^U=(X\!\times\!\{x\})(U\!\times\!\{x\})=(X\!\times\!\{x\})\{(x,x)\}=X\!\times\!\{x\}=\Xi_x\,.
$$ 
This example also shows another difference between the group and the groupoid case. For groups, the compact set ${\sf K}$ in the definition of syndeticity can always be taken finite (see for example \cite[pag.271]{dV}, where this property is called 'discrete syndeticity'). In this groupoid no finite set ${\sf K}\subset X\times \{x\}$ makes the equality ${\sf K}\,\widetilde\Xi_x^U\!=\Xi_x$ true if $X$ itself is infinite. 
\end{ex} 

If $\Si$ is a compact space, Zorn's Lemma implies that it has a minimal subset $M\subset\Si$\,. If in addition $\Xi$ is open, every $x\in M$ is almost periodic, by Theorem \ref{sacacarez}. Thus, in this setting, almost periodic points always exist.

\begin{cor}
If $\,\Xi$ is open (and locally compact), the set $\Si_{\rm alper}$ is invariant.
\end{cor}

\begin{proof}
The second part of Theorem \ref{sacacarez} guarantees that, if $\si\in\Si_{\rm alper} $\,, then $\mathfrak O_\si\subset\Si_{\rm alper}$\,.
\end{proof}

\begin{cor}\label{sindecat}
Suppose that $\Si$ is compact, $\Xi$ is open and the action is minimal. Then $\widetilde\Xi_\si^V$ is syndetic for every $\si\in\Si$ and every open non-void subset $V$ of $\,\Si$ (and not just for neighborhoods of $\si$)\,.  
\end{cor}

\begin{proof}
By minimality, there exists $\zeta\in\Xi_{\rho(\si)}$ such that $V$ is an open neighborhood of $\zeta\bu\si$. Hence $\widetilde\Xi_{\zeta\bu\si}^V$ is syndetic by Theorem \ref{sacacarez}; it can be written as ${\sf K}\,\widetilde\Xi_{\zeta\bu\si}^V\!=\Xi_{\rho(\zeta\bu\si)}$ for some compact subset ${\sf K}$ of $\Xi$\,. Then 
$$
\Xi_{\rho(\si)}=\Xi_{\d(\zeta)}=\Xi_{\r(\zeta)}\,\zeta=\Xi_{\rho(\zeta\bu\si)}\,\zeta={\sf K}\,\widetilde\Xi_{\zeta\bu\si}^V\,\zeta={\sf K}\,\widetilde\Xi_\si^V,
$$ 
meaning that $\,\widetilde\Xi_\si^V$ is also syndetic.
\end{proof}

We say that the action is {\it semisimple} if all the orbit closures are minimal (equivalently: the orbit closures form a partition of $\Si$)\,.

\begin{cor}\label{oblu}
If all the orbits are closed, the action is semisimple. A pointwise almost periodic action is semisimple.  If $\,\Xi$ is open and all the orbits are compact, the action is pointwise almost periodic.
\end{cor}

\begin{proof}
The statements are obvious or they follow easily from Theorem \ref{sacacarez}. 
\end{proof}

\begin{rem}\label{inversus}
In Example \ref{valtoare} the orbits are precisely the $\d$-fibers, automatically closed: $\mathfrak O_\eta=\overline{\mathfrak O}_\eta=\Xi_{\d(\eta)}$\,. In particular this action of $\Xi$ on itself to the left is semisimple. Pointwise almost periodicity relies on compactness of the fibers. 
\end{rem}

\begin{prop}\label{zdrikat}
If $\,\Si$ is compact and minimal and $\Xi$ is open and strongly non-compact, the action is non-wandering.
\end{prop}

\begin{proof}
We know from Proposition \ref{haihui} and Corollary \ref{prinurmare} that $\Si_{\rm nw}$ is a non-void closed invariant set. By minimality, it coincides with $\Si$\,. 
\end{proof}

\section{Factors}\label{kalorfier}

\subsection{Epimorphisms of groupoid actions}\label{ficatorel}

We indicate in this subsection the fate of some of the properties above under epimorphisms. Some of the results do not use surjectivity; we leave this to the reader.

\begin{defn}
 An {\it homomorphism} of the groupoid actions $(\Xi,\rho,\th,\Si)$\,, $(\Xi,\rho',\th',\Si')$ is a continuous function $f:\Si\rightarrow\Si'$ such that for all $\si\in\Si\,,\xi\in\Xi_{\rho(\si)}$ one has
 \begin{equation}\label{requirement}
\rho'\big(f(\si)\big)= \rho(\si)\quad{\rm and}\quad f\big(\th(\xi,\si)\big)=\th'\big(\xi,f(\si)\big)\,.
 \end{equation}
 An {\it epimorphism} is a surjective homomorphism; in such a case we say that $(\Xi,\rho',\th',\Si')$ is {\it a factor} of $(\Xi,\rho,\th,\Si)$ and that $(\Xi,\rho,\th,\Si)$ is {\it an extension} of $(\Xi,\rho',\th',\Si')$\,.
 \end{defn}
 
Writing $\bu$ instead of $\th$ and $\bu'$ instead of $\th'$, the second requirement in \eqref{requirement} is 
 \begin{equation}\label{restat}
 f(\xi\bu\si)=\xi\bu'\!f(\si)\,,\quad\forall\,(\xi,\si)\in\Xi\!\Join\!\Si\,.
 \end{equation}
 
 \begin{lem}\label{ajuta}
 The canonical action $(\Xi,{\rm id}_X,\circ,X)$ from Example \ref{startlet} is a factor of any other continuous action $\,(\Xi,\rho,\bu,\Si)$\,.
 \end{lem}
 
\begin{proof}
The continuous surjection $f\!:=\rho:\Si\to X$ is an epimorphism, since it satisfies
\begin{equation*}
\rho(\xi\bu\si)=\r(\xi)=\xi\circ\d(\xi)=\xi\circ\rho(\si)\,,\quad\forall\,(\xi,\si)\in\Xi\!\Join\!\Si\,.
\end{equation*}
\end{proof}

\begin{lem}\label{despreorbite}
If $f$ is an epimorphism between the groupoid dynamical systems $(\Xi,\rho,\th,\Si)$ and $(\Xi,\rho',\th',\Si')$ and $\si\in\Si$\,, then $f(\mathfrak O_\si)=\mathfrak O'_{f(\si)}$ and $f\big(\overline{\mathfrak O}_\si\big)\subset\overline{\mathfrak O}'_{f(\si)}$\,. Direct images of invariant sets are also invariant.
\end{lem}

\begin{proof}
The elementary proof relies on \eqref{restat} and on the properties of continuous functions. We recall that for continuous surjections the direct image may not commute with the closure.
\end{proof}

\begin{prop}\label{gagiu}
For every homomorphism $f$ between actions $(\Xi,\rho,\bu,\Si)$\,, $\big(\Xi,\rho',\bu',\Si'\big)$ and every point $\si\in\Si$\,, we have  $f\big(\mathfrak L_\si^{\th}\big)=\mathfrak L_{f(\si)}^{\th'}$\,. In consequence, $f(\si)$ is $\th'$-recurrent if $\si$ is $\th$-recurrent.
\end{prop}

\begin{proof}
By Proposition \ref{labil}, $\tau'\in f\big(\mathfrak L_\si^\th\big)$ if and only if exists a divergent net $(\eta_i)_{i\in I}$ in $\Xi_{\rho(\si)}=\Xi_{\rho'[f(\si)]}$ such that 
$$
\tau'=f(\lim_i \eta_i\bu\si)=\lim_i f(\eta_i\bu\si)=\lim_i \eta_i\bu' f(\si)\,,
$$ 
which is equivalent with $\tau'\in \mathfrak L_{f(\si)}^{\th'}$\,. Then the statement about recurrence follows from the definitions.
\end{proof}

Let us see what happens with recurrence sets under epimorphisms.
 
 \begin{lem}\label{label}
 If $M,N\subset\Si$ then $\widetilde\Xi_{M}^{N}\subset\widetilde\Xi\,_{f(M)}^{\prime f(N)}$\,, where the later set is computed with respect to the factor $(\Xi,\rho',\th',\Si')$ of the groupoid dynamical system $(\Xi,\rho,\th,\Si)$\,. 
 \end{lem}
 
\begin{proof}
One verifies easily that the next sequence of equivalences and implications is rigorous:
$$
\begin{aligned}
\xi\in\widetilde\Xi_{M}^{N}&\Leftrightarrow (\xi\bu M)\cap N\ne\emptyset\Leftrightarrow f\big[(\xi\bu M)\cap N\big]\ne\emptyset\\
&\Rightarrow f(\xi\bu M)\cap f(N)\ne\emptyset\Leftrightarrow \big(\xi\bu'\!f(M)\big)\cap f(N)\ne\emptyset\\
&\Leftrightarrow \xi\in\widetilde\Xi\,_{f(M)}^{\prime f(N)}\,.
\end{aligned}
 $$
 The second equivalence is true because $f$ is onto. In general one has $f(A\cap B)\subset f(A)\cap f(B)$ and the inclusion could be strict; this shows why (and when) there is no equality in the statement.
 \end{proof}
 
 To see the usefulness of this Lemma, we hurry to apply it. On many occasions we are going to use the equality $f\big[f^{-1}(B')\big]=B'$ for $B'\subset\Si'$, valid by surjectivity.
 
 \begin{thm}\label{chiroare}
 Let $f$ be an epimorphism between the groupoid actions $(\Xi,\rho,\bu,\Si)$ and $\big(\Xi,\rho',\bu',\Si'\big)$\,. For every index $\alpha\in\{{\rm fix,per,wper,alper,rec,nw}\}$ one has $f\big(\Si_{\alpha}\big)\subset\Si'_{\alpha}$\,. In particular, $\rho\big(\Si_{\alpha}\big)\subset X_{\alpha}$\,, where $X_\alpha$ indicates the set of units of $X$ having the property $\alpha$ with respect to the canonical action.
 \end{thm}
 
 \begin{proof}
 Using Lemma \ref{ajuta} the last statement follows from the first, that we now prove.
 
 \smallskip
 For $\alpha={\rm rec}$ this is already known from Proposition \ref{gagiu}. The statement for $\alpha={\rm fix}$ follows immediately from \eqref{requirement}. 
 
 \smallskip
 We prove now the case $\alpha={\rm nw}$\,. Assume that $\si\in\Si$ is non-wandering, but $f(\si)\in\Si'$ is wandering. Then there is an open neighborhood $W'$ of $f(\si)$ such that $\widetilde\Xi_{\,W'}^{\prime W'}$ is relatively compact. The set $f^{-1}(W')$ is a neighborhood of $\si$, so $\widetilde\Xi_{f^{-1}(W')}^{f^{-1}(W')}$ is not relatively compact. By Lemma \ref{label}, this set is contained in $\widetilde\Xi_{\,W'}^{\prime W'}$, and this is a contradiction.
 
 \smallskip
The other proofs are similar.  For instance, the statement about almost periodicity follows easily from the definitions and from Lemma \ref{label}; a set containing a syndetic subset is obviously syndetic.
 \end{proof}
 
\begin{prop}\label{morptranz}
Let $f$ be an epimorphism between the groupoid actions $(\Xi,\rho,\bu,\Si)$ and $\big(\Xi,\rho',\bu',\Si'\big)$\,. Suppose that the action $(\Xi,\rho,\bu,\Si)$ has one of the transitivity properties 
$$
\mathcal P\in\{\textup{{\it transitivity}},\, \textup{{\it pointwise transitivity}},\textup{{\it weak pointwise transitivity}},(i),(ii),(iii)\}
$$
(see Theorem \ref{pricinoasa}). Then the action $\big(\Xi,\rho',\bu',\Si'\big)$ also has $\mathcal P$.
\end{prop}

\begin{proof}
$\mathcal P=\textup{{\it transitivity}}.$ By Lemma \ref{despreorbite}, the epimorphism $f$ transforms the single orbit $\mathfrak O_\si=\Si$ into an orbit $\mathfrak O'_{f(\si)}=f(\mathfrak O_\si)=\Si'$.

\smallskip
$\mathcal P=\textup{{\it pointwise transitivity}}.$ This also follows immediately from Lemma \ref{despreorbite}, the part referring to orbit closures.

\smallskip
$\mathcal P=\textup{\it weak pointwise transitivity}.$ If $C'\subset\Xi'$ is a closed and invariant set containing $f(\si)$, then $f^{-1}(C')\subset \Xi$ is a closed and invariant set containing $\si$, hence $\Si=\mathfrak C_\si\subset f^{-1}(C')$ and $\Si'=C'$.

\smallskip
$\mathcal P=(i)\Leftrightarrow(i').$ Let $\emptyset\ne U',V'\subset \Si'$ be invariant open sets and let $f^{-1}(U')\,,f^{-1}(V')$ be their open non-void invariant inverse images. Since $(\Xi,\rho,\bu,\Si)$ satisfies $(i')$, one has $f^{-1}(U')\cap f^{-1}(V')\ne\emptyset$\,. Consequently
$$
U'\cap V'\!=f\big[f^{-1}(U')]\cap f\big[f^{-1}(V')\big]\supset f\big[f^{-1}(U')\cap f^{-1}(V')\big]\ne\emptyset
$$
and $\big(\Xi,\rho',\bu',\Si'\big)$ also satisfies $(i')$\,, which is equivalent with $(i)$.

\smallskip
$\mathcal P=(ii).$ Let $U'\subset \Si'$ be open and invariant. Then $f^{-1}(U')\subset\Si$ is open and invariant, so it is dense. Since $f$ is surjective it follows that $U'=f\big(f^{-1}(U')\big)$\,, and this one is dense in $\Si'$.

\smallskip
$\mathcal P=(iii).$ Let $\emptyset\ne U',V'\subset\Si'$ be open subsets. By recurrent transitivity of the initial action, one has $\widetilde\Xi^{f^{-1}(V')}_{f^{-1}(U')}\ne\emptyset$\,. Then, by Lemma \ref{label}, we get
$$
\emptyset\ne \widetilde\Xi^{f^{-1}(V')}_{f^{-1}(U')}\subset\widetilde\Xi^{\prime f[f^{-1}(V')]}_{\,f[f^{-1}(U')]}=\widetilde\Xi^{\prime V'}_{\,U'}
$$
and the proof is finished.
\end{proof}

\begin{cor}
We say that {\rm the groupoid $\Xi$ has the property $\mathcal P$} if its canonical action on its unit space has this property. Suppose that the topological groupoid $\Xi$ admits a continuous action $(\rho,\th,\Si)$ having one of the properties $\mathcal P$ mentioned in Proposition \ref{morptranz}. Then $\Xi$ itself has this property.
\end{cor}

\begin{proof}
This follows from Proposition \ref{morptranz} and Lemma \ref{ajuta}. 
\end{proof}

If we require $\Xi$ to be open, there is a direct proof that the property $(iv)$ from Theorem \ref{pricinoasa} (called topological transitivity) also transfer to factors; it uses Lemma \ref{intclo}. This also follows joining Theorem \ref{pricinoasa} and Proposition \ref{morptranz}. But see the next example for a non-topologically transitive groupoid, which however possesses a topologically transitive action. It follows that topological transitivity is not preserved by factors (which is rather surprising).

\begin{ex}\label{valioso4} 
Form the product groupoid $\Xi=\Pi\times \R$\,, with the relation 
$$
x\,\Pi\,y\ \Leftrightarrow\ x,y\in\mathbb Q\ \ \textup{or}\ \ x=y
$$ 
over $X=\R$\,, and consider the wide subgroupoid 
$$
\Delta=\{(x,y,g)\in\Xi\mid x,y\not\in\mathbb Q\Rightarrow g=0\}=\big(\mathbb Q\!\times\!\mathbb Q\!\times\!\R\big)\cup\big({\rm Diag}(\R\!\times\!\R)\!\times\!\{0\}\big)\,.
$$ 
By analogy with Example \ref{valioso3}, its easy to see that the canonical action of $\Delta$ on $X=\R$ is not topologically transitive. Actually, the orbits of rational points all coincide with $\mathbb Q$\,, but each irrational point is a fixed point, so $(s,t)\!\setminus\!\mathbb Q$ is invariant for every $s<t$\,, without being dense or nowhere dense. Now we will exhibit a topologically transitive action of $\Xi$\,: Let 
$$
\Si=\{(y,h)\in\R^2\mid h=0\textup{ or }y\in\mathbb Q\}=(\R\!\times\!\{0\})\cup(\mathbb Q\!\times\!\R)
$$ 
with the topology inherited from $\R^2$ and define the continuous action
$$
\rho(y,h)=y\quad \textup{ and }\quad (x,y,g)\bu(y,h)=(x,g+h)\,.
$$
Notice that the orbits of $\bu$ are $\mathbb Q\!\times\!\R$ and (the fixed points) $\{(y,0)\!\mid\!y\in\R\setminus\mathbb Q\}$\,, implying that all of the invariant sets are either dense or nowhere dense, since $\Si\!\setminus\!(\mathbb Q\!\times\!\R)=(\R\!\setminus\!\mathbb Q)\!\times\!\{0\}$ is already nowhere dense.
\end{ex}

We finish this subsection with a result on the behavior of minimality under epimorphisms, in both directions.

\begin{prop}\label{garbanzos}
Let $f:\Si\to\Si'$ be an epimorphism between the actions $(\Xi,\rho,\th,\Si)$ and $\big(\Xi,\rho',\th',\Si'\big)$\,. 
\begin{enumerate}
\item[(i)]
If $M\subset\Si$ is minimal and $f(M)$ is closed in $\Si'$, then $f(M)$ is minimal. 
\item[(ii)]
Suppose that $\Si$ is compact (hence $\Si'$ is also compact). If $M'\subset\Si'$ is minimal, there exists $M\subset\Si$ minimal such that $f(M)=M'$. If $\,\Xi$ is open and locally compact and $\si'\in\Si'$ is almost periodic, then $\si'=f(\si)$ for some almost periodic point $\si$ of $\,\Si$\,.
\end{enumerate}
\end{prop}

\begin{proof}
(i) This is dealt with easily by Lemma \ref{despreorbite}: if $\si\in M$ then
$$
f(M)=f\big(\overline{\mathfrak O}_\si\big)\subset\overline{\mathfrak O}'_{f(\si)}\subset\overline{f(M)}=f(M)\,,
$$ 
so the orbit of $f(\si)$ is dense in the closed set $f(M)$\,.

\smallskip
(ii) The inverse image $f^{-1}(M')$ is non-void closed and $\bu$-invariant. By Zorn's Lemma, it contains a minimal (and compact) subsystem $M$. The direct image $f(M)\subset M'$ is non-void closed and $\bu'$-invariant, so it must coincide with $M'$.  For the second part: As $\si'$ is almost periodic, $M'\!:=\overline{\mathfrak O}'_{\si'}$ is minimal, so we can find some (compact) minimal subset $M\subset\Si$\,, such that $f(M)=M'$ and $\si'=f(\si)$ for some $\si\in M$. By Theorem \ref{sacacarez}, any $\si\in M$ is almost periodic.
\end{proof}

\subsection{The action associated to an extension}\label{fregadorel}

The next result is a straightforward generalization of a recent construction from \cite{EK}, in which the authors showed how to encode extensions of classical group actions by groupoids. Then we are going to study the resulting correspondence with respect to recurrence and other dynamical properties.

\begin{prop}
Let $(\Xi,\rho',\th',\Si')$ be a groupoid action. There is a one-to-one correspondence between extensions $f:\Si\to\Si'$ and actions of the groupoid $\,\Xi\!\ltimes_{\th'}\!\Si'$.
\end{prop}

\begin{proof}
Let $(\Xi,\rho,\th,\Si)\overset{f}{\to}(\Xi,\rho',\th',\Si')$ be an epimorhism. We denote by $\,\Xi(\th')=\Xi\!\ltimes_{\th'}\!\Si'$ the action groupoid of the second groupoid action, described in Subsection \ref{castanier}. Its unit space is $\Si'$. This gives raise to a new groupoid action $\big(\Xi(\th'),f,\Th,\Si\big)$\,, where (see the notation \eqref{ganchor} and use $\rho'\circ f=\rho$)
\begin{equation*}\label{cambridge}
\begin{aligned}
\Xi(\th')\!\Join\!\Si:&=\big\{\big((\xi,\si'),\si\big)\,\big\vert\,\d(\xi)=\rho'(\si')\,,\,f(\si)=\si'\big\}\\
&=\big\{\big(\xi,f(\si),\si\big)\,\big\vert\,\d(\xi)=\rho'[f(\si)]=\rho(\si)\big\}\subset\Xi\!\times\Si'\!\times\!\Si
\end{aligned}
\end{equation*}
and
\begin{equation}\label{oxford}
\Th_{(\xi,f(\si))}(\si):=\th_\xi(\si)\,,\quad{\rm if}\ \d(\xi)=\rho(\si)\,.
\end{equation}
It is straightforward to check that everything is well-defined and that one obtains a groupoid action. 

\smallskip
Reciprocally, suppose that $(\Xi,\rho',\th',\Si')$ is a groupoid action and $\big(\Xi(\th'),f,\Th,\Si\big)$ is an action of the transformation groupoid $\Xi(\th')=\Xi\!\ltimes_{\th'}\!\Si'$ on another topological space $\Si$\,. In particular, $f:\Si\to\Si'=\Xi(\th')^{(0)}$ is a continuous surjection. The new action $(\Xi,\rho,\th,\Si)$ is constructed in the following way: The anchor is 
$$
\rho:=\rho'\circ f:\Si\to X\!:=\Xi^{(0)}
$$
and the action is given by \eqref{oxford}, with a modified interpretation: the r.h.s. is now defined by the l.h.s. Then $f:\Si\to\Si'$ is the epimorphism.
The details are easy. To show that $f\circ\th_\xi=\th_\xi'\circ f$, for instance, pick $\si\in\Si$ with $\d(\xi)=\rho(\si)=\rho'[f(\si)]$\,. Then
$$
f\big[\th_\xi(\si)\big]\overset{\eqref{oxford}}{=}f\big[\Th_{(\xi,f(\si))}(\si)\big]={\sf r}\big(\xi,f(\si)\big)\overset{\eqref{garucho}}{=}\th'_\xi\big[f(\si)\big]\,.
$$
Here ${\sf r}$ is the range map of $\Xi\!\ltimes_{\th'}\!\Si'$\,.
The two procedures are inverse to each other.
\end{proof}

\begin{rem}\label{jenctura}
It is easy to see that the projection
$$
\Xi\!\times\!\Si'\!\times\!\Si\ni(\xi,\si'\!,\si)\to(\xi,\si)\in\Xi\!\times\!\Si
$$
restricts to a groupoid isomorphism $\big(\Xi\!\ltimes_{\th'}\!\Si'\big)\!\ltimes_\Th\Si\cong\Xi\!\ltimes_{\th}\!\Si$\,. The basic remark is that the elements of $\big(\Xi\!\ltimes_{\th'}\!\Si'\big)\!\ltimes_\Th\Si$ have the form $((\xi,f(\si)),\si)\equiv(\xi,f(\si),\si)$\,, with $\d(\xi)=\rho(\si)$\,. The algebraic verifications are left to the reader.
\end{rem}

\begin{prop}\label{vindeo}
The two actions $(\Xi,\rho,\th,\Si)$ and $\big(\Xi(\th'),f,\Th,\Si\big)$ have the same invariant sets. Consequently, they have in the same time the following properties: transitivity, point transitivity, topological transitivity, the properties (i) and (ii) of Theorem \ref{pricinoasa} and minimality. The fixed points are the same.
\end{prop}

\begin{proof}
Let us explore the orbit structure; for $\si,\tau\in\Si$ one has
$$
\begin{aligned}
\si\overset{\Th}{\sim}\tau&\ \Leftrightarrow\;\exists\,(\xi,\si')\in\Xi(\th')\ \,{\rm such\ that}\,\ \si'=f(\si)\,,\,\Th_{(\xi,\si')}(\si)=\tau\\
&\ \Leftrightarrow\;\exists\,\xi\in\Xi\ \,{\rm such\ that}\,\ \d(\xi)=\rho(\si)\,,\,\th_\xi(\si)=\tau\\
&\ \Leftrightarrow\;\si\overset{\th}{\sim}\tau.
\end{aligned}
$$
All the statements of the Proposition follow from this.
\end{proof}

Now, we discuss recurrence for the connected actions $\big(\Xi(\th'),f,\Th,\Si\big)$ and $(\Xi,\rho,\th,\Si)$. 
For $M,N\subset\Si$ we have two recurrent sets: $\widetilde{\Xi(\th')}^N_M$ (with respect to the action $\Th$) and $\widetilde{\Xi}^N_M$ (with respect to the action $\th$)\,. Using the definitions one gets
\begin{equation*}\label{falafel}
\begin{aligned}
(\xi,\si')\in\widetilde{\Xi(\th')}^N_M&\;\Leftrightarrow\;\d(\xi)=\rho'(\si')\ {\rm and}\ \exists\,\si\in M\,,\,f(\si)=\si'\ \,{\rm such\ that}\,\ \Th_{(\xi,\si')}(\si)\in N\\
&\ \Leftrightarrow\;\exists\,\si\in M,\,f(\si)=\si',\,\d(\xi)=\rho(\si)\ \,{\rm such\ that}\,\ \th_\xi(\si)\in N,
\end{aligned}
\end{equation*}
while
\begin{equation*}\label{fafalel}
\xi\in\widetilde{\Xi}^N_M\ \Leftrightarrow\;\exists\,\si\in M\,,\,\d(\xi)=\rho(\si)\ \,{\rm such\ that}\,\ \th_\xi(\si)\in N.
\end{equation*} 
Then it is clear that
\begin{equation*}\label{farfalel}
\xi\in\widetilde{\Xi}^N_M\Leftrightarrow\;\exists\,\si'\in\Si'\ {\rm such\ that}\ (\xi,\si')\in\widetilde{\Xi(\th')}^N_M\,, 
\end{equation*}
which can be written in terms of the first projection $p:\Xi\!\times\!\Si'\!\to\Xi$ as
\begin{equation}\label{farfurie}
p\Big[\widetilde{\Xi(\th')}^N_M\Big]=\widetilde{\Xi}^N_M\,.
\end{equation} 
Setting $M=\{\si\}$\,, we get the more precise version 
\begin{equation}\label{rticular}
\widetilde{\Xi(\th')}^N_\si=\widetilde{\Xi}^N_\si\!\times\!\{f(\si)\}\,.
\end{equation} 
Two particular cases are 
\begin{equation}\label{retrotic}
{\Xi(\th')}_{\rho(\si)}={\Xi}_{\rho(\si)}\!\times\!\{f(\si)\}
\end{equation} 
and
\begin{equation}\label{retrotac}
\widetilde{\Xi(\th')}^\si_\si\!=\widetilde{\Xi}^\si_\si\!\times\!\{f(\si)\}\,,
\end{equation}
obtained by setting in \eqref{rticular} $N=\Si$ or $N=\{\si\}$\,. With this, we can state the following proposition: 

\begin{prop}\label{deruta}
Consider the actions $\big(\Xi(\th'),f,\Th,\Si\big)$ and $(\Xi,\rho,\th,\Si)$\,. The action $\big(\Xi(\th'),f,\Th,\Si\big)$ is recurrently transitive if and only if $(\Xi,\rho,\th,\Si)$ is so.  One has the equalities:
\begin{equation}\label{strikat}
\Si^{\Th}_{\rm per}=\Si^\th_{\rm per}\,,\quad\Si^{\Th}_{\rm wper}=\Si^\th_{\rm wper}\,,\quad\Si^{\Th}_{\rm alper}=\Si^\th_{\rm alper}\,,\quad\Si^{\Th}_{\rm rec}=\Si^\th_{\rm rec}\,.
\end{equation} 
\end{prop}

\begin{proof}
The statement about recurrent transitivity follows from \eqref{farfurie}, taking $M=U, N=V$ nonvoid open sets, since $p(R)\ne\emptyset$ if and only if $R\ne\emptyset$\,.

The equality $\Si^{\Th}_{\rm wper}=\Si^\th_{\rm wper}$ follows from \eqref{retrotac}: $\widetilde{\Xi(\th')}^\si_\si$ is compact if and only if $\widetilde{\Xi}^\si_\si$ is compact. 

\smallskip 
To prove $\Si^{\Th}_{\rm per}=\Si^\th_{\rm per}$\,, let ${\sf K}\subset\Xi$ be a compact set with ${\sf K}\,\widetilde\Xi_\si^\si=\Xi_{\rho(\si)}$\,. Then ${\sf K}'={\sf K}\!\times\!\{f(\si)\}\subset\Xi(\th')$ is compact and 
$$
{\sf K}'\,\widetilde{\Xi(\th')}^\si_\si\overset{\eqref{retrotac}}{=}{\sf K}'\big(\widetilde{\Xi}^\si_\si\!\times\!\{f(\si)\}\big)=\big({\sf K}\widetilde{\Xi}^\si_\si\big)\!\times\!\{f(\si)\}=\Xi_{\rho(\si)}\!\times\!\{f(\si)\}\overset{\eqref{retrotic}}{=}\Xi(\th')_{\rho(\si)}\,.
$$ 
The second equality is consequence of a computation of the form
$$
(\xi,f(\si))(\eta,f(\si))=(\xi\eta,f(\si))\,,\quad {\rm for}\ \;\eta\in \widetilde{\Xi}^\si_\si\,,\ \;{\rm implying}\ \;\eta\bu'\!f(\si)=f(\si)\,.
$$ 
So syndeticity of $\widetilde{\Xi}^\si_\si$ implies the syndeticity of $\widetilde{\Xi(\th')}^\si_\si$\,. On the other hand, if $\,\widetilde{\Xi(\th')}^\si_\si$ is syndetic with compact set ${\sf K}'$, to obtain $\sf K$ one takes the projection onto the first coordinate: ${\sf K}=p({\sf K}')$\,. The same type of argument works to show that $\Si^{\Th}_{\rm alper}=\Si^\th_{\rm alper}$\,.

\smallskip
The relation $\Si^{\Th}_{\rm rec}=\Si^\th_{\rm rec}$ holds because of \eqref{rticular}, by using Proposition \ref{recur}.
\end{proof}

\begin{rem}\label{gomplement}
Proving \eqref{strikat}, we took advantage of the fact that $\Si_{\rm per}\,,\Si_{\rm wper}\,,\Si_{\rm alper}\,,\Si_{\rm rec}$ are defined by recurrent sets involving a couple of subsets of $\Si$\,, one of them being a singleton; this allowed the application of \eqref{rticular}. One also gets $\Si^{\Th}_{\rm nw}=\Si^\th_{\rm nw}$ if $\Si'$ is supposed compact, using \eqref{farfurie}. Without compactness, only one obvious inclusion is true.
\end{rem}

\section{A disappointing notion: mixing}\label{frigider}

The next definition seems a legitimate generalization of the classical one:

\begin{defn}\label{ixing}
The action $(\Xi,\rho,\th,\Si)$ is called {\it weakly mixing} whenever for every $U,U',V,V'\!\subset\Si$ non-empty open sets, one has $\widetilde\Xi_U^V\cap \widetilde\Xi_{U'}^{V'}\!\ne\emptyset$\,. It is called {\it strongly mixing} if the complement of $\widetilde\Xi_U^V$ is relatively compact for every open sets $U,V\ne\emptyset$\,.
\end{defn}

{\it If $\,\Xi$ itself is not compact, strongly mixing implies weakly mixing}. This follows from the equality
$$
\big(\widetilde\Xi_U^V\cap \widetilde\Xi_{U'}^{V'}\big)^c=\big(\widetilde\Xi_U^V\big)^c\cup \big(\widetilde\Xi_{U'}^{V'}\big)^c
$$
and the fact that the union of two relatively compact sets is relatively compact. On the other hand, {\it weakly mixing always implies recurrent transitivity}: just take $U=U'$ and $V=V'$. 

\begin{ex}\label{safacut}
For classical group actions one recovers the usual concepts; see Example \ref{bunul}.
\end{ex}

The next result shows that there in no point in exploring the notion of mixing besides Example \ref{safacut} (the group case), which is already extensively treated in all the standard textbooks in Topological Dynamics. Recall that we assumed $\rho$ surjective and $X$ Hausdorff.

\begin{prop}
Suppose that $(\Xi,\rho,\theta,\Sigma)$ is weakly mixing. Then $\rho(\Sigma)=X$ consists of a single point, so the groupoid is a group.
\end{prop}

\begin{proof}
Suppose that exists two distinct points $x,y\in\rho(\Sigma)$\,, and let $U_0,V_0\subset X$ be disjoint open sets separating them. Form the non-void open sets $U=\rho^{-1}(U_0)$\,, $V=\rho^{-1}(V_0)$\,. By weakly mixing, there exists some $\xi\in\widetilde\Xi_{U}^{U}\cap \widetilde\Xi_{U}^{V}$. Thus, there are points $\sigma_1,\sigma_2,\sigma_3\in U,\tau\in V$ such that 
$$
\xi\bullet \sigma_1=\sigma_2\quad\textup{and}\quad \xi\bullet \sigma_3=\tau,
$$ 
implying that 
$$
\rho(\sigma_2)=\rho(\xi\bullet \sigma_1)=\r(\xi)=\rho(\xi\bullet \sigma_3)=\rho(\tau)\,.
$$
But $\rho(\sigma_2)\in U_0$ and $\rho(\tau)\in V_0$\,. This contradiction shows that $\rho(\Sigma)$ consist of a single point.
\end{proof}


\bigskip
\bigskip
ADDRESS

\smallskip
F. Flores:

Departamento de Ingenier\'ia Matem\'atica, Universidad de Chile, 

Beauchef 851, Torre Norte, Oficina 436,

Santiago, Chile

E-mail: feflores@dim.uchile.cl

\smallskip
M. M\u antoiu:

Departamento de Matem\'aticas, Universidad de Chile, 

Las Palmeras 3425, Casilla 653, 

Santiago, Chile 

E-mail: mantoiu@uchile.cl


\begin{thebibliography}{1}

\bibitem{Ab} F. Abadie: \textit{On Partial Actions and Groupoids}, Proc. Amer. Math. Soc. \textbf{132}(4), 1037--1047, (2004).

\bibitem{Au} J. Auslander: \textit{Minimal Flows and Their Extensions}, North Holland, Amsterdam, 1988.

\bibitem{BBdN} S. Beckus, J. Bellissard and G. de Nittis: \emph{Spectral Continuity for Aperiodic Quantum Systems I. General theory}, Journal of Functional Analysis, \textbf{275}, 2917--2977,  (2018).




\bibitem{CKN} G. Cairns, A Kolganova and A. Nielsen: \emph{Topological Transitivity and Mixing Notions for Group Actions}, Rocky Mountain J. of Math. \textbf{37}(2), 371--397, (2007).

\bibitem{CNQ}  C. Carvalho, V. Nistor, and Y. Qiao: \textit{Fredholm Conditions on Non-compact Manifolds: Theory and Examples}, in Operator Theory, Operator Algebras, and Matrix Theory, volume 267 of Oper. Theory Adv. Appl., pages 79--122. Birkh\"auser/Springer, Cham, 2018.

\bibitem{De} V. Deaconu: \emph{Groupoids Associated with Endomorphisms}, Trans. Amer. Math. Soc. \textbf{347}, 1779--1786, (1995).

\bibitem{dV} J. de Vries: \textit{Elements of Topological Dynamics}, in \textit{Mathematics and its Applications}, Kluwer Acad. Publ., Dordrecht, 1993.

\bibitem{EK} N. Edeko and H. Kreidler: \textit{Uniform Enveloping Semigroupoids for Groupoid Actions}, Preprint ArXiV.

\bibitem{EE} D.\,B. Ellis and R. Ellis: \textit{Automorphisms and Equivalence Relations in Topological Dynamics}, London Mathematical Society Lecture Notes Series, \textbf{412}, 2013.

\bibitem{Ex} R. Exel: \textit{Inverse Semigroups and Combinatorial $C^*$-Algebras}, Bull. Braz. Math. Soc. (N.S.) \textbf{39}(2), 191--313, (2008).

\bibitem{Ex1} R. Exel: \textit{Partial Dynamical Systems, Fell Bundles and Applications}, Mathematical Surveys and Monographs, \textbf{224}, American Mathematical Society, 2017.

\bibitem{Gl} E. Glasner: \textit{Ergodic Theory via Joinings}, Mathematical Surveys and Monographs, \textbf{101}, American Mathematical Society, 2003.


\bibitem{HM} Ph. Higgins and K. Mackenzie: \textit{Algebraic Constructions in the Category of Lie Algebroids}, J. Algebra, \textbf{129}(1), 194--230, (1990).


\bibitem{Mac} K. Mackenzie: \textit{Lie Groupoids and Lie Algebroids in Differential Geometry}, volume 124 of LMS Lect. Note Series. Cambridge U. Press, Cambridge, 1987

\bibitem{Pal} R.S. Palais: \textit{On the Existence of Slices for Actions of Non-compact Lie Groups}, Ann. of Math. \textbf{73}, 295--323, (1961).

\bibitem{Pa} A.\,L\,T. Paterson: \textit{Groupoids, Inverse Semigroups, and Their Operator Algebras}, Progress in Mathematics, \textbf{170}, Birkh\"auser Boston, Inc., Boston, MA, 1999.


\bibitem{Ra} A. Ramsay: \textit{The Mackey-Glimm Dichotomy for Foliations and Other Polish Groupoids}, J. Funct. Anal. \textbf{94}, 358--374, (1990).

\bibitem{Re} J. Renault: \emph{A Groupoid Approach to $C^*$-Algebras} Lecture Notes in Mathematics, \textbf{793}, Springer, Berlin, 1980.

\bibitem{St} B. Steinberg: \emph{Prime \'Etale Groupoid Algebras with Applications to Inverse Semigroup and Leavitt Path Algebras}, J. Pure Appl. Algebra, \textbf{223}(6), 2474--2488, (2019). 

\bibitem{Wi} D Williams: \textit{A Tool Kit for Groupoid C*-Algebras}, Mathematical Surveys and Monographs, \textbf{241}, AMS, Providence, Rhode Island, 2019.

\end{thebibliography}
\end{document}